\documentclass[a4paper, 11pt]{amsart}
\usepackage[utf8]{inputenc}
\usepackage[english]{babel}
\usepackage[T1]{fontenc}
\usepackage{amsmath, amssymb, amsfonts}
\usepackage{enumerate}
\usepackage{graphicx}
\usepackage{comment}
\usepackage[hmargin=1.5cm, vmargin=2.5cm]{geometry}
\usepackage{tikz}
\usepackage{hyperref}

\title{Wreath products of finite groups by quantum groups}
\author{Amaury Freslon}
\author{Adam Skalski}
\address{Amaury Freslon, Laboratoire de Math\'ematiques d'Orsay, Univ. Paris-Sud, CNRS, Universit\'e Paris-Saclay, 91405 Orsay, France}
\address{Adam Skalski,  Institute of Mathematics of the Polish Academy of Sciences, ul. \'{S}niadeckich 8, 00-656 Warszawa, Poland}
\email{amaury.freslon@math.u-psud.fr}
\email{a.skalski@impan.pl}
\date{}

\theoremstyle{plain}
\newtheorem{thm}{Theorem}[section]
\newtheorem{prop}[thm]{Proposition}
\newtheorem{cor}[thm]{Corollary}
\newtheorem{lem}[thm]{Lemma}

\theoremstyle{definition}
\newtheorem{de}[thm]{Definition}

\theoremstyle{remark}
\newtheorem{rem}[thm]{Remark}
\newtheorem{question}[thm]{Question}

\DeclareMathOperator{\ii}{id}
\DeclareMathOperator{\Id}{Id}
\DeclareMathOperator{\Ir}{Irr}
\DeclareMathOperator{\Mor}{Hom}

\DeclareMathOperator{\Span}{Span}
\DeclareMathOperator{\Stab}{Stab}
\DeclareMathOperator{\rl}{rl}
\DeclareMathOperator{\Tr}{Tr}

\newcommand{\bo}{\boldsymbol}

\renewcommand{\tilde}{\widetilde}

\newcommand{\A}{\mathcal{A}}

\newcommand{\C}{\mathbb{C}}
\newcommand{\CC}{\mathcal{C}}
\newcommand{\D}{\Delta}

\newcommand{\G}{\mathbb{G}}

\newcommand{\N}{\mathbb{N}}

\newcommand{\U}{\mathcal{U}}
\newcommand{\V}{\mathcal{V}}
\newcommand{\W}{\mathcal{W}}
\newcommand{\Z}{\mathbb{Z}}

\newcommand{\h}{\widehat}

\newcommand{\crosspart}{
\mathrel{\offinterlineskip
\hbox{$/$}\hskip -1.10ex\hbox{$\backslash$}}}

\begin{document}

\begin{abstract}
We introduce a notion of partition wreath product of a finite group by a partition quantum group, a construction motivated on the one hand by classical wreath products and on the other hand by the free wreath product of J. Bichon. We identify the resulting quantum group in several cases, establish some of its properties and show that when the finite group in question is abelian, the partition wreath product is itself a partition quantum group. This allows us to compute its representation theory, using earlier results of the first named author.
\end{abstract}

\maketitle

\section{Introduction}

When quantum groups first appeared in mathematics in the 1980s (with some developments obviously possible to be traced back to much earlier days), they were described via their "algebras of functions". Soon after that, S.L.\,Woronowicz developed in \cite{woronowicz1988tannaka} his extension of the Tannaka-Krein duality, showing that a compact quantum group can be fully described via its representation theory viewed as a certain $C^{*}$-tensor category (see the book \cite{neshveyev2014compact} for a precise description). This development, far from being purely of theoretical interest, made it possible to construct new examples of quantum groups and study properties of the ones that had already been known. In particular, starting with the article \cite{banica2009liberation} it has become clear that an important role is played in this context by the categories of partitions, which led to the introduction of the so-called easy quantum groups. In fact, the use of the combinatorics of partitions in representation theory has a long history with origins in classical works of R.\,Brauer and H.\,Weyl from the early twentieth century. For the description of these developments and also for a full characterization of partition quantum groups -- of which in a sense canonical examples are the free permutation groups $S_{N}^{+}$ of S.\,Wang (\cite{wang1998quantum}) -- we refer to a recent preprint \cite{freslon2014partition} of the first-named author.

One of the well-known constructions in classical group theory is that of a wreath product, a specific instance of a semidirect product based on the action of a given group $H$ on several copies of another group $G$ by permutations. One often assumes from the beginning that $H$ itself is a group of permutations, as this does not form an essential constraint. In \cite{bichon2004free} J.\,Bichon generalised this construction to the quantum group setting, replacing $G$ by an arbitrary compact matrix quantum group $\G$ and $H$ by a quantum permutation group $\mathbb{H}$, that is a quantum subgroup of a free permutation group of Wang.  This "free" wreath product, $\G \wr_{*} \mathbb{H}$, defined in an algebraic manner, via its "algebra of functions" and the related fundamental representation, has turned out to have several interesting properties, which were later studied by T.\,Banica, J.\,Bichon, F.\,Lemeux, P.\,Tarrago and others (see for example the recent preprint \cite{lemeux2014free}). In particular, the fusion rules -- i.e.\ essentially the representation theory -- of $\G \wr_{*} S_{N}^{+}$ were computed for all compact matrix quantum groups $\G$ of Kac type. One needs to note that the construction of J.\,Bichon, although clearly inspired by the classical wreath product, does not reduce to it even in the case when $\G$ is a classical group and $\mathbb{H}=S_{N}$.

The aim of this paper is the study of a new construction inspired by the above-mentioned works, which we call \emph{partition wreath product}. Instead of defining it through a universal C*-algebra, as J.\,Bichon did for his free wreath products in \cite{bichon2004free}, we choose in a sense a converse path, defining it via the associated C*-tensor category, in the spirit of \cite{banica2009liberation}. The input data consists of a finite group $G$ and a category of partitions $\CC$. The associated  category is built by averaging the morphisms coming from $\CC$ using the group $G$. In fact, there is strong evidence (in particular coming from \cite{lemeux2014free}) that such a construction can be extended, with $G$ replaced by an arbitrary compact (quantum) group. The resulting quantum group is denoted by $G \wr \G_{N}(\CC)$, where $\G_{N}(\CC)$ denotes the $N$-th partition quantum group associated to the category of partitions $\CC$. It is worth noting that this approach to quantum wreath products is suggested by classical results of \cite{bloss2003g} and \cite{parvathi2004g}.

In the particular cases when $\CC$ is the category of all partitions (respectively of all non-crossing partitions),  $\G_{N}(\CC)$ is the permutation group $S_{N}$ (respectively the free permutation quantum group $S_{N}^{+}$) and we recover the usual wreath product $G \wr S_{N}$ (respectively the free wreath product $G \wr_{*} S_{N}^{+}$). We are also able to identify some further cases, so that for example for the quantum hyperoctahedral groups $H_{N}^{+}$ we have $G \wr H_{N}^{+} \simeq (G \times \Z_{2}) \wr_{*} S_{N}^{+}$. A key tool allowing us to study the algebraic and probabilistic properties of the partition wreath product is that of a sudoku representation, introduced in the study of the family of quantum hyperoctahedral groups $H_{N}^{s+}$ (isomorphic to  free wreath product $\Z_{s} \wr_{*} S_{N}^{+}$) by T.\,Banica and R.\,Vergnioux in \cite{banica2009fusion}.

The detailed plan of the paper is as follows: in Section \ref{sec:preliminaries} we introduce some background on compact quantum groups and partitions. We then define the partition wreath product in Section \ref{sec:construction} and study its basic properties. In particular we present there the \emph{sudoku picture} and compute the law of the character of the fundamental representation. We then turn in Section \ref{sec:abelian} to the case of abelian finite groups $G$ for which  we show that the resulting  partition wreath product is itself a partition quantum group in the sense of \cite{freslon2014partition}.  This allows us to use the results of \cite{freslon2014partition} to fully describe the representation theory of $G \wr \G_{N}(\CC)$. Eventually, we introduce in Section \ref{sec:generalisation} some examples of possible generalisations of our setting and discuss certain open questions arising from this work.

\subsection*{Acknowledgements}

The first author was partially supported by the ERC advanced grant "Noncommutative distributions in free probability". The second author was partially supported by the NCN grant 2014/14/E/ST1/00525.

\section{Preliminaries}\label{sec:preliminaries}

In this preliminary section, we gather material on compact quantum groups and the combinatorics of partitions which will be used throughout the paper. Our aim is to give the necessary definitions and results in a concise way, as well as to fix the notations. References for details will be given in the text. All along the paper, scalar products will always be left-linear.

\subsection{Partitions and linear maps}

Our main tool in this work will be partitions of finite sets. The use of these for the study of representation theory has a long history, relying on a particular graphical representation (see for instance \cite{freslon2014partition} for some references). Let $P(k, l)$ be the set of partitions of the set $\{1, 2, \dots, k+l\}$. We represent such partitions in the following way: we draw a line of $k$ points above a line of $l$ points and then connect the points which belong to the same subset of the partition.
This pictorial description makes it in particular easy to work on the \emph{blocks} of the partitions, which we now define.

\begin{de}
Let $p$ be a partition.
\begin{itemize}
\item A maximal set of points which are all connected (i.e. one of the subsets defining the partition) is called a block of $p$.
\item If $b$ contains both upper and lower points (i.e. the subset contains an element of $\{1, \dots, k\}$ and an element of $\{k+1, \dots, k+l\}$), then it is called a \emph{through-block}.
\item Otherwise, it is called a \emph{non-through-block}.
\end{itemize}
We will write $b\subset p$ if $b$ is a block of $p$.
\end{de}

The total number of blocks of $p$ is denoted by $b(p)$ and its number of through-blocks is denoted by $t(p)$. In the present paper, we will be particularly interested in \emph{non-crossing partitions}.

\begin{de}
Let $p$ be a partition. A \emph{crossing} in $p$ is a tuple $k_{1} < k_{2} < k_{3} < k_{4}$ of integers such that:
\begin{itemize}
\item $k_{1}$ and $k_{3}$ are in the same block.
\item $k_{2}$ and $k_{4}$ are in the same block.
\item The four points are \emph{not} in the same block.
\end{itemize}
If there is no crossing in $p$, then it is said to be  a \emph{non-crossing} partition. The set of non-crossing partitions with $k$ upper points and $l$ lower points will be denoted by $NC(k, l)$.
\end{de}

The set of all partitions can be endowed with several operations:
\begin{itemize}
\item  The \emph{tensor product} of two partitions $p\in P(k, l)$ and $q\in P(k', l')$ is the partition $p\otimes q\in P(k+k', l+l')$ obtained by \emph{horizontal concatenation}, i.e. the first $k$ of the $k+k'$ upper points are connected by $p$ to the first $l$ of the $l+l'$ lower points, while $q$ connects the remaining $k'$ upper points to the remaining $l'$ lower points.
\item The \emph{composition} of two partitions $p\in P(l, m)$ and $q\in P(k, l)$ is the partition $pq\in P(k, m)$ obtained by \emph{vertical concatenation}: we connect $k$ upper points by $q$ to $l$ middle points and then continue the lines by $p$ to $m$ lower points. This yields a partition connecting $k$ upper points with $m$ lower points. By the composition procedure, certain loops might appear resulting from blocks around the middle points. More precisely, consider the set $L$ of elements in $\{1, \dots, l\}$ which are not connected to a lower point of $p$ nor to an upper point of $q$. The upper row of $p$ and the lower row of $q$ both induce partitions of the set $L$. The maximum (with respect to inclusion of blocks) of these two partitions is the \emph{loop partition} of $L$, its blocks are called \emph{loops} and their number is denoted by $\rl(p, q)$. To finish the operation, we remove all the loops in order to produce a partition in $P(k, m)$.
\item The \emph{involution} of a partition $p\in P(k, l)$ is the partition $p^{*}\in P(l, k)$ obtained by flipping $p$ upside down.
\item We also have a \emph{rotation} on partitions. Let $p\in P(k, l)$ be a partition connecting $k$ upper points with $l$ lower points. Rotating the leftmost upper point to the left of the lower row (or the converse) gives rise to a partition in $P(k-1, l+1)$ (or in $P(k+1, l-1)$), called a \emph{rotated version} of $p$. Rotation may also be performed on the right.
\end{itemize}
These operations are called the \emph{category operations}. We will be interested in collections of partitions which are stable under these operations.

\begin{de}\label{de:categroypartitions}
A collection $\CC$ of subsets $\CC(k, l)\subseteq P(k, l)$ (for every $k, l\in\N$) is a \emph{category of partitions} if it is stable under all the category operations and if the \emph{identity partition} $\vert\in P(1, 1)$ is in $\CC(1, 1)$.
\end{de}

The interplay between partitions and quantum groups is based on a natural way of associating linear maps to partitions.

\begin{de}
Let $N\geqslant 1$ be an integer and let $(e_{1}, \dots, e_{N})$ be a basis of $\C^{N}$. For any partition $p$, we define a linear map
\begin{equation*}
T_{p}:(\C^{N})^{\otimes k} \mapsto (\C^{N})^{\otimes l}
\end{equation*}
by the following formula:
\begin{equation*}
T_{p}(e_{i_{1}} \otimes \dots \otimes e_{i_{k}}) = \sum_{j_{1}, \dots, j_{l} = 1}^{n} \delta_{p}((i_{1}, \dots, i_{k}), (j_{1}, \dots, j_{l}))e_{j_{1}} \otimes \dots \otimes e_{j_{l}},
\end{equation*}
where $\delta_{p}((i_{1}, \dots, i_{k}), (j_{1}, \dots, j_{l})) = 1$ if and only if all the strings of the partition $p$ connect equal indices of the multi-index $(i_{1}, \dots, i_{k})$ in the upper row with equal indices of the multi-index $(j_{1}, \dots, j_{l})$ in the lower row. Otherwise, $\delta_{p}((i_{1}, \dots, i_{k}), (j_{1}, \dots, j_{l})) = 0$.
\end{de}

We will have to do several computations involving multi-indices as in the formula above. We therefore introduce some notations to simplify these expressions. We will use bold letters  to denote multi-indices, for instance $\bo{i} = (i_{1}, \dots, i_{k})$. We also set $e_{\bo{i}} = e_{i_{1}} \otimes \dots \otimes e_{i_{k}}$. The definition of the map $T_{p}$ then reads
\begin{equation*}
T_{p}(e_{\bo{i}}) = \sum_{\bo{j}}\delta_{p}(\bo{i}, \bo{j}) e_{\bo{j}}.
\end{equation*}

The interplay between the category operations on partitions and the assignment $p \mapsto T_{p}$ was studied by T. Banica and R. Speicher in \cite[Prop. 1.9]{banica2009liberation}. It can be summarized as follows:

\begin{prop}\label{prop:compositionTp}
The assignment $p\mapsto T_{p}$ satisfies
\begin{enumerate}
\item $T_{p}^{*} = T_{p^{*}}$
\item $T_{p}\otimes T_{q} = T_{p\otimes q}$
\item $T_{p}T_{q} = N^{\rl(p, q)}T_{pq}$
\end{enumerate}
\end{prop}

The case when the partitions are non-crossing will be important later on because of the following Proposition (see for instance \cite[Lem 4.16]{freslon2013representation} for a proof):

\begin{prop}\label{prop:linearindependence}
Let $N\geqslant 4$ be an integer and let $k, l\in \N$. Then, the linear maps $(T_{p})_{p\in NC(k, l)}$ are linearly independent.
\end{prop}

\subsection{Compact quantum groups}

The aim of this work is to use partitions to produce compact quantum groups as defined by S.L. Woronowicz in \cite{woronowicz1995compact}. We therefore recall some basic definitions and results of this theory. 

\begin{de}
A \emph{compact quantum group} is a pair $\G = (C(\G), \D)$ where $C(\G)$ is a unital C*-algebra,
\begin{equation*}
\D: C(\G)\rightarrow C(\G)\otimes C(\G)
\end{equation*}
is a unital $*$-homomorphism such that $(\D\otimes \ii)\circ\D = (\ii \otimes \D)\circ \D$, and the linear spans of $\D(C(\G))(1\otimes C(\G))$ and $\D(C(\G))(C(\G)\otimes 1)$ are dense in $C(\G)\otimes C(\G)$ (all the tensor products of C*-algebras are spatial).
\end{de}

Compact quantum groups have a rich and well-behaved representation theory. In particular, finite-dimensional representations carry enough information to recover the whole quantum group, which is the reason why we will focus on these.

\begin{de}
Let $\G$ be a compact quantum group and let $n\in \N$. A \emph{representation} of $\G$ of dimension $n$ is a matrix $(u_{ij})_{1\leqslant i, j\leqslant n}\in M_{n}(C(\G)) \simeq C(\G) \otimes M_{n}(\C)$ such that
\begin{equation*}
\D(u_{ij}) = \sum_{k=1}^{n} u_{ik}\otimes u_{kj}
\end{equation*}
for every $1 \leqslant i, j \leqslant n$. Moreover,
\begin{itemize}
\item the representation $u$ is said to be \emph{unitary} if it is a unitary element of $M_{n}(C(\G))$.
\item the \emph{contragredient} $\overline{u}$ of $u$ is the representation defined by $\overline{u}_{ij} = u_{ij}^{*}$.
\end{itemize}
\end{de}

An \emph{intertwiner} between two representations $u$ and $v$ of dimension respectively $n$ and $m$ is a linear map $T: \C^{n} \rightarrow \C^{m}$ such that
\begin{equation*}
(\ii \otimes T)u = v(\ii \otimes T).
\end{equation*}
The set of intertwiners between $u$ and $v$ is denoted by $\Mor_{\G}(u, v)$, or simply $\Mor(u, v)$ if there is no ambiguity. If there exists a unitary intertwiner between $u$ and $v$, then $u$ and $v$ are said to be \emph{unitarily equivalent}. A representation $u$ is said to be \emph{irreducible} if $\Mor(u, u) = \C.\ii$. The \emph{tensor product} of two representations $u$ and $v$ is the representation
\begin{equation*}
u\otimes v = u_{12}v_{13}\in C(\G)\otimes M_{n}(\C)\otimes M_{m}(\C) \simeq C(\G)\otimes M_{nm}(\C),
\end{equation*}
where we used the \emph{leg-numbering} notation: for an operator $X$ acting on a twofold tensor product, $X_{ij}$ is the extension of $X$ acting on the $i$-th and $j$-th tensors of a triple tensor product. As for compact groups, it is possible to reconstruct a compact quantum group from a C*-tensor category equipped with some additional structure. This is S.L. Woronowicz's Tannaka-Krein theorem proved in \cite{woronowicz1988tannaka}. In particular, starting with a category of partitions, it is possible to build such a C*-tensor category. The associated compact quantum group will then be completely determined by the combinatorics of the initial set of partitions. Let us state this result in the spirit of \cite[Prop 3.12]{banica2009liberation}.

\begin{thm}
Let $N\geqslant 1$ be an integer and let $\CC$ be a category of partitions. Then, there exists a compact quantum group $\G$ together with a finite-dimensional unitary representation $u$ such that
\begin{itemize}
\item The coefficients of $u$ generate a dense subalgebra of $C(\G)$
\item For any $k, l\in \N$, $\Mor_{\G}(u^{\otimes k}, u^{\otimes l}) = \Span\{T_{p}, p\in \CC(k, l)\}$
\end{itemize}
Moreover, the group $\G$ is unique up to isomorphism and is denoted by $\G_{N}(\CC)$.
\end{thm}

Let us mention that the quantum groups defined by this theorem, as well as all the ones that we are going to study in this work, are of \emph{Kac type} (see \cite[Thm 1.5]{woronowicz1995compact} for a list of equivalent characterizations of this property).

\begin{rem}
We will need several times a notion of quantum subgroup -- if $\G$ and $\mathbb{H}$ are compact quantum groups, then we say that $\mathbb{H}$ is a (closed) quantum subgroup of $\G$ if there exists a unital surjective $*$-homomorphism from $C(\G)$ onto $C(\mathbb{H})$ intertwining the respective coproducts (precisely speaking we require that the morphism acts on the universal level -- for a detailed discussion we refer for
example to \cite{daws2012closed}).
\end{rem}

\subsection{Wreath products}\label{subsec:wreath}

As explained in the introduction, we will study generalisations of the wreath product construction. We therefore briefly recall the classical construction, as well as its free (quantum) counterpart due to J.\,Bichon.

\begin{de}
Let $G$ be a group and let $N\geqslant 1$ be an integer. The \emph{(permutational) wreath product} of $\G$ by $S_{N}$ is the semi-direct product $G\wr S_{N} = G^{N}\rtimes S_{N}$, where the action of $S_{N}$ on $G^{N}$ is given by permutation of the factors.
\end{de}

The construction above can be extended by replacing $S_N$ by any group acting on the set $\{1,\ldots,N\}$ by permutations. There are several examples of groups which can be decomposed as wreath products. In particular the complex reflection groups $G(s, 1, N)$ are isomorphic to $\Z_{s}\wr S_{N}$. The main drawback of this definition from a quantum group perspective is the appearance of the semi-direct product, for which there is no good analogue when $S_{N}$ is replaced by a compact quantum group. To remedy this, note that abstractly speaking, an element of $G\wr S_{N}$ may be seen as the product of a permutation matrix and a diagonal matrix with coefficients in $G$, the matrix product then inducing the group law. To make this more precise, we will resort to a specific "permutation representation" of $G\wr S_{N}$. Let us consider the vector space $V$ spanned by vectors $(e_{i}^{g})_{1\leqslant i\leqslant N, g\in G}$. For each $1\leqslant i\leqslant N$, we define a representation $\rho_{i}$ of $G$ by
\begin{equation*}
\rho_{i}(g)e_{j}^{h} = \left\{
\begin{array}{ccc}
e_{j}^{gh} & \text{if} & i = j \\
e_{j}^{h} & \text{if} & i \neq j
\end{array}
\right.
\end{equation*}
for any $1\leqslant j\leqslant N$ and $h\in G$. We also define a representation $\pi$ of $S_{N}$ on $V$ by $\pi(\sigma)(e_{i}^{g}) = e_{\sigma^{-1}(i)}^{g}$, $1\leqslant i\leqslant N, g \in G$. This gives us a new characterization of the wreath product.

\begin{prop}
The subgroup of $GL(V)$ generated by $\pi(S_{N})$ and $\rho_{i}(G)$ for $1\leqslant i\leqslant N$ is isomorphic to the wreath product $G\wr S_{N}$.
\end{prop}

\begin{proof}
Noticing that the representations $\rho_{i}$ and $\rho_{j}$ commute when $i\neq j$, we can identify the subgroup generated by the images of $\rho_{i}(G)$ for all $1\leqslant i\leqslant N$ with the group of diagonal $N\times N$ matrices with coefficients in $G$. Since $\pi$ is just the permutation representation of $S_{N}$, the result is clear.
\end{proof}

\begin{rem}
The proposition above suggests a way of defining a wreath product of a group $G$ by an arbitrary linear group $H\subset GL(N)$ ($N \in \mathbb{N}$). Indeed, we can repeat the construction above and define the "linear" wreath product $G \wr H$ to be the subgroup of $GL(V)$ generated by $\pi(H)$ and $\rho_{i}(G)$ for $1\leqslant i\leqslant N$, where $\pi$ denotes an appropriate amplification of the identity representation of $H$. The above proposition implies that when $H \subset S_N$, then we obtain the usual (permutational) wreath product. The precise relation of this construction with the partition wreath products we introduce in the next section is not clear.
\end{rem}

Schur-Weyl duality for wreath products by $S_{N}$ has a combinatorial description using partitions,  introduced by M. Bloss in \cite{bloss2003g}. This description was reformulated in a setting closer to ours by A.J.\,Kennedy and M.\,Parvathi in \cite{parvathi2004g}. There is also a quantum analogue of the wreath product, introduced by J. Bichon in \cite{bichon2004free}, which we now describe. It involves the \emph{free (or quantum) permutation group} $S_{N}^{+}$ defined by S.\,Wang in \cite{wang1998quantum}.

\begin{de}
Let $\G$ be a compact quantum group with a fundamental representation $u$ and let $v$ be the fundamental representation of $S_{N}^{+}$. Consider the free product $C(\G)^{\ast N}\ast C(S_{N}^{+})$ and let, for $1\leqslant k\leqslant N$, $\nu_{k}$ be the inclusion of $C(\G)$ as the $k$-th factor of the free product. The \emph{free wreath product algebra} is the quotient of $C(\G)^{\ast N}\ast C(S_{N}^{+})$ by the relations
\begin{equation*}
\nu_{i}(a)v_{ij} = v_{ij}\nu_{i}(a)
\end{equation*}
for all $a\in C(\G)$ and $1\leqslant i, j\leqslant N$. By \cite[Thm 3.2]{bichon2004free}, the free wreath product algebra admits a natural compact quantum group structure, called the \emph{free wreath product} of $\G$ by $S_{N}^{+}$ and denoted by $\G\wr_{\ast}S_{N}^{+}$.
\end{de}

The representation theory of free wreath products was studied by F. Lemeux and P. Tarrago in \cite{lemeux2014free} using partition methods.

\section{Partition wreath product}\label{sec:construction}

In this section we introduce our construction of partition wreath products. We will give several characterizations of these objects so as to be able to study them and compare them to known constructions. The basic ingredients for the construction are a finite group $G$ and a category of partitions $\CC$ and we will therefore use these symbols without further reference. We will moreover denote by $e$ the neutral element of $G$.

\subsection{Averaged partitions}

The construction starts by colouring the partitions in $\CC$ with the elements of $G$. This means that we now consider partitions with the additional data of an element of $G$ attached to each point. The set of all such partitions will be denoted by $\CC^{G}$.

\begin{de}
Let $k, l\in \N$ and let $p\in \CC^{G}(k, l)$. The \emph{upper colouring} of $p$ is the sequence $\bo{c}_{u}(p)$ of colours of the upper points of $p$, read from left to right. Similarly, the \emph{lower colouring} of $p$ is the sequence $\bo{c}_{d}(p)$ of colours of the lower points, also read from left to right.
\end{de}

The category operations can be extended to $\CC^{G}$ with the only constraint that two partitions $p$ and $q$ can only be composed (in this order) if the lower colouring of $q$ is the same as the upper colouring of $p$. We then get a \emph{category of coloured partitions} in the sense of \cite[Def 3.1.2]{freslon2014partition}, but we will not need this general framework for the moment. The assignment $p\mapsto T_{p}$ can be extended to this setting in the following way: fix a copy $V^{g}$ of $\C^{N}$ for each $g\in G$. Then, if $p\in \CC^{G}$ has upper colouring $g_{1}, \dots, g_{k}$ and lower colouring $h_{1}, \dots h_{l}$, the same formula as for the uncoloured case yields a map
\begin{equation*}
T_{p} : V^{g_{1}}\otimes \dots\otimes V^{g_{k}} \rightarrow V^{h_{1}}\otimes \dots\otimes V^{h_{l}}.
\end{equation*}
We are now going to use the action of $G$ on itself to produce new objects. More precisely, let $p\in \CC^{G}$ and order its blocks $b_{1}, \dots, b_{b(p)}$ according to their leftmost point. Then, a tuple $(g_{1}, \dots, g_{b(p)})$ of elements of $G$ acts on $p$ by multiplying the colours of all the points of $b_{i}$ by $g_{i}$ on the left, for all $1\leqslant i\leqslant b(p)$. We denote this new partition by $(g_{1}, \dots, g_{b(p)}).p$. If there is an element $g\in G$ such that $g_{i} = g$ for all $1\leqslant i\leqslant b(p)$, then we are simply multiplying all the colours of $p$ by $g$ on the left, and the result is denoted by $g.p$.

\begin{de}\label{de:averagedoperators}
Let $N\geqslant 1$ be an integer. For $p\in \CC^{G}$, we define the \emph{averaged operators} $L_{p}$ and $M_{p}$ by
\begin{eqnarray*}
L_{p} & = & \sum_{g\in G}T_{g.p} \\
M_{p} & = & \sum_{(g_{1}, \dots, g_{b(p)})\in G^{b(p)}}T_{(g_{1}, \dots, g_{b(p)}).p}.
\end{eqnarray*}
\end{de}

One may expect the sum to be divided by the number of elements in the definition above. However, this would make the formula for the composition much more complicated, so we will rather use the above expression. As for the operators $T_{p}$, our aim is to produce a C*-tensor category with the averaged operators, so that there is an associated compact quantum group. We would like to build categories $\mathfrak{C}_{\times}(\CC, G, N)$ and $\mathfrak{C}_{\wr}(\CC, G, N)$ whose objects are non-negative integers and whose morphisms are
\begin{eqnarray*}
\Mor_{\mathfrak{C}_{\times}(\CC, G, N)}(k, l) & = & \Span\{L_{p}, p\in \CC^{G}(k, l)\} \\
\Mor_{\mathfrak{C}_{\wr}(\CC, G, N)}(k, l) & = & \Span\{M_{p}, p\in \CC^{G}(k, l)\} \\
\end{eqnarray*}
The first step in analysing these categories is to understand the composition of morphisms. For operators of the form $L_{p}$, this was done in \cite[Lem 4.2.9]{freslon2014partition}. For operators of the form $M_{p}$, the computation is more involved. Let us introduce some notations. If $p\in \CC(k, l)$ and if $\bo{g} = (g_{1}, \dots, g_{k})$ and $\bo{h} = (h_{1}, \dots, h_{l})$ are tuples of elements of $G$, we denote by $p(\bo{g}, \bo{h})$ the partition obtained by colouring the lower points of $p$ by $\bo{g}$ (from left to right) and the upper points of $p$ by $\bo{h}$ (also from left to right). We also write $\bo{u}\leqslant p$ to indicate that $\bo{u} = (u_{1}, \dots, u_{b(p)})$ is a tuple indexed by the blocks of $p$.

\begin{lem}\label{lem:compositionaveraged}
Let $p\in P(l, m)$, $q\in P(k, l)$ and let $\bo{g}, \bo{h}, \bo{s}, \bo{t}$ be suitable tuples of elements of $G$. If there exist tuples $\bo{x} = (x_{1}, \dots, x_{b(p)})$ and $\bo{y} = (y_{1}, \dots, y_{b(q)})$ of elements of $G$ such that $\bo{c}_{u}(\bo{x}.p(\bo{g}, \bo{h})) = \bo{c}_{d}(\bo{y}.q(\bo{s}, \bo{t}))$ then
\begin{equation*}
M_{p(\bo{g}, \bo{h})}M_{q(\bo{s}, \bo{t})} = (\vert G\vert N)^{\rl(p, q)}M_{pq(\bo{c}_{d}(\bo{x}.p(\bo{g}, \bo{h})), \bo{c}_{u}(\bo{y}.q(\bo{s}, \bo{t})))}.
\end{equation*}
Otherwise, $M_{p(\bo{g}, \bo{h})}M_{q(\bo{s}, \bo{t})} = 0$.
\end{lem}

\begin{proof}
Using the definition of the averaged operators, we have
\begin{equation*}
M_{p(\bo{g}, \bo{h})}M_{q(\bo{s}, \bo{t})} = N^{\rl(p, q)}\sum_{\bo{u}\leqslant p, \bo{v}\leqslant q}\delta(\bo{c}_{u}(\bo{u}.p(\bo{g}, \bo{h})), \bo{c}_{d}(\bo{v}.q(\bo{s}, \bo{t})))T_{pq(\bo{c}_{d}(\bo{u}.p(\bo{g}, \bo{h})), \bo{c}_{u}(\bo{v}.q(\bo{s}, \bo{t})))}.
\end{equation*}
The condition on $\bo{x}$ and $\bo{y}$ in the statement means precisely that
\begin{equation*}
\delta(\bo{c}_{u}(\bo{x}.p(\bo{g}, \bo{h})), \bo{c}_{d}(\bo{y}.q(\bo{s}, \bo{t}))) \neq 0.
\end{equation*}
If no such tuples $\bo{x}$ and $\bo{y}$ exist, then the $\delta$-symbol always vanishes and the result is $0$. Otherwise, let $\bo{x}'$ and $\bo{y}'$ be other tuples for which the $\delta$-symbol does not vanish and let us consider an integer $1\leqslant i\leqslant l$. The point $i$ in the upper row of $p$ belongs to a certain block $b$, while the point $i$ in the lower row of $q$ belongs to another block $d$. Since we know that the $\delta$-symbol is non-zero, the colours of $i$ are the same in both partitions. In other words, $x_{b}h_{i} = y_{d}s_{i}$ and $x'_{b}h_{i} = y'_{d}s_{i}$, implying that $x'_{b}x_{b}^{-1} = y'_{d}y_{d}^{-1}$. This equality holds for any blocks $b$ and $d$ having a common point, hence also for any blocks $b$ and $d$ satisfying the following conditions:
\begin{itemize}
\item $b$ contains an upper point of $p$
\item $d$ contains a lower point of $q$
\item $b$ and $d$ get connected when composing $p$ and $q$.
\end{itemize}
Therefore, we can find two tuples $\bo{w} = (w_{1}, \dots, w_{b(p)})$ and $\bo{z} = (z_{1}, \dots, z_{b(q)})$ which coincide on  blocks of $p$ and $q$ which get connected in the composition and such that $\bo{x}' = \bo{w}\bo{x}$ and $\bo{y}' = \bo{z}\bo{y}$. Reciprocally, any such tuples $\bo{w}$ and $\bo{z}$ yield tuples $\bo{x}'$ and $\bo{y}'$ such that the $\delta$-symbol is non-zero. Summarizing, we can write
\begin{equation*}
M_{p(\bo{g}, \bo{h})}M_{q(\bo{s}, \bo{t})} = N^{\rl(p, q)}\sum_{\bo{w}, \bo{z}}T_{pq(\bo{c}_{d}(\bo{w}\bo{x}.p(\bo{g}, \bo{h})), \bo{c}_{u}(\bo{z}\bo{y}.q(\bo{s}, \bo{t})))}.
\end{equation*}
The data of the tuples $\bo{w}$ and $\bo{z}$ is almost the same as a tuple of colours corresponding to the blocks of $pq$. More precisely, let us associate to each pair $(\bo{w}, \bo{z})$ the unique tuple $\bo{a} = (a_{1}, \dots, a_{b(pq)})$ which matches $\bo{w}$ and $\bo{z}$ in a natural way. Then, two pairs give the same $\bo{a}$ if and only if they only differ on upper non-through-blocks of $p$ and lower non-through-block of $q$. But since we know that $\bo{w}$ and $\bo{z}$ are constant on all blocks which are connected, they only differ by an element of $G$ on each loop of the composition. Thus, each tuple $\bo{a}$ will appear $\vert G\vert^{\rl(pq)}$ times in the sum. We are then left with a sum over colourings of the blocks of $pq$, hence the result.
\end{proof}

\begin{prop}
The categories $\mathfrak{C}_{\times}(\CC, G, N)$ and $\mathfrak{C}_{\wr}(\CC, G, N)$ are concrete C*-tensor categories with duals.
\end{prop}

\begin{proof}
This was proved in \cite[Lem 4.2.9]{freslon2014partition} for $\mathfrak{C}_{\times}(P, G, N)$. For $\mathfrak{C}_{\wr}(P, G, N)$ the same strategy works for the horizontal concatenation, the involution and the rotation. The only difficulty is the stability under composition, which was established in Lemma \ref{lem:compositionaveraged}.
\end{proof}

By S.L. Woronowicz's Tannaka-Krein theorem, each of these concrete C*-tensor categories is the category of finite-dimensional representations of a compact quantum group. To understand the resulting compact quantum groups, we will compare them to those arising in the wreath product construction. So let us fix an integer $N\in \N$ and consider the vector space $V$ spanned by elements $(e_{i}^{g})_{1\leqslant i\leqslant N, g\in G}$. There is a natural representation $\rho$ of $G$ on $V$ given by
\begin{equation*}
\rho(g)e_{j}^{h} = e_{j}^{gh}.
\end{equation*}
Seeing $V$ as $\C^{N}\otimes \C^{\vert G\vert}$, we note that $\rho = \ii\otimes \lambda_{G}$, where $\lambda_{G}$ denotes the left regular representation of $G$. We will also consider the representations $(\rho_{i})_{1\leqslant i\leqslant N}$ defined in Subsection \ref{subsec:wreath}. These representations can be used to characterize the averaged operators.

\begin{prop}\label{prop:equivariance}
Let $k, l$ be two integers and consider an operator $X \in\Span\{T_{p}, p\in P^{G}(k, l)\}$. Then,
\begin{itemize}
\item $\rho(g^{-1})^{\otimes l}\circ X\circ \rho(g)^{\otimes k} = X$ for all $g\in G$ if and only if $X\in \Span\{L_{p}, p\in P^{G}(k, l)\}$.
\item $\rho_{i}(g^{-1})^{\otimes l}\circ X\circ \rho_{i}(g)^{\otimes k} = X$ for all $1\leqslant i \leqslant N$ and $g\in G$ if and only if $X\in \Span\{M_{p}, p\in P^{G}(k, l)\}$.
\end{itemize}
\end{prop}

\begin{proof}
The first assertion is a direct consequence of \cite[Lem 4.2.8]{freslon2014partition} and the second one was proved in \cite[Lem 6.1]{bloss2003g}.

\end{proof}

\begin{rem}\label{rem:conjugationinvariance}
It was proved in \cite[Prop 4.3.1]{freslon2014partition} that for any category of partitions $\CC$, the set of linear combinations of operators $T_{p}$ for $p\in \CC$ which are invariant under conjugation by $\rho$ is precisely the span of the operators $L_{p}$ for $p\in \CC$. When considering conjugation by $\rho_{i}$ instead, we can only deduce from Proposition \ref{prop:equivariance} that the set of invariant elements is $\Span\{M_{p}, p\in P^{G}(k, l)\}\cap\Span\{T_{p}, p\in \CC^{G}(k, l)\}$. Whether this space is the same as $\Span\{M_{p}, p\in \CC^{G}(k, l)\}$ remains an open problem.
\end{rem}

Using Proposition \ref{prop:equivariance}, we can identify some of the compact quantum groups associated to the C*-tensor categories defined above.

\begin{cor}\label{cor:classicalpermutation}
For any category of partitions $\CC$, the compact quantum group associated to $\mathfrak{C}_{\times}(\CC, G, N)$ is the direct product $G\times \G_{N}(\CC)$. Moreover, the compact group associated to $\mathfrak{C}_{\wr}(P, G, N)$ is the wreath product $G\wr S_{N}$.
\end{cor}

\begin{proof}
The first assertion was proved in \cite[Prop 4.3.1]{freslon2014partition}. The second statement is a consequence of \cite[Thm 4.1.4]{parvathi2004g} and the results of \cite{bloss2003g}.
\end{proof}

The last  result justifies the notations $\mathfrak{C}_{\times}$ and $\mathfrak{C}_{\wr}$, as well as the following definition.

\begin{de}
Let $N\geqslant 1$ be an integer and let $\CC$ be a category of partitions. We denote by $G\wr \G_{N}(\CC)$ the compact quantum group associated to $\mathfrak{C}_{\wr}(\CC, G, N)$ and call it the \emph{partition wreath product} of $G$ by $\G_{N}(\CC)$.
\end{de}

\begin{rem}
It is clear that the fundamental representation of $G\wr\G_{N}(\C)$ is invariant under conjugation by the representations $\rho_{i}$ and $\pi$. The question of Remark \ref{rem:conjugationinvariance} is whether partition wreath products are the only quantum groups satisfying this property.
\end{rem}

The partition wreath product has links to the free wreath products of J. Bichon, even though the two constructions are different in essence. This will be easier to see once we have some results on the structure of the C*-algebra $C(G\wr\G_{N}(\CC))$, which is the reason why we postpone this discussion to the end of Subsection \ref{subsec:permutation}.

\subsection{Sudoku matrices}

We now want to give a description of $G\wr\G_{N}(\CC)$ as a compact matrix quantum group. This means that we want to describe a fundamental representation of $G\wr\G_{N}(\CC)$ as a matrix with coefficients in $C(G\wr\G_{N}(\CC))$, and provide a description of these coefficients as generators of a universal $*$-algebra satisfying certain algebraic conditions. The main idea is a generalisation of the sudoku matrices introduced in \cite{banica2009fusion}.

\begin{de}
Let $N\geqslant 1$ be an integer. A \emph{$(G, N)$-sudoku matrix} is a matrix $U \in M_{N\times \vert G\vert}(A)$, for some C*-algebra $A$, of the form $[Q_{h^{-1}g}]_{g, h\in G}$ for matrices $Q_{g}\in M_{N}(A)$.
\end{de}

\begin{rem}
In the case $G = \Z_{s}$, this definition is equivalent to the one of a $(s, N)$-sudoku matrix given in \cite[Def 2.2]{banica2009fusion}.
\end{rem}

Let $\CC$ be a category of partitions and let $N\geqslant 1$ be an integer. Then, $C(\G_{N}(\CC))$ is the universal C*-algebra generated by a $N\times N$ unitary matrix $U$ whose coefficients satisfy relations which are given by the partitions in $\CC$, and therefore do not depend on $N$. Let us call these relations the \emph{$\CC$-relations}.

\begin{de}\label{de:sudoku}
Let $N\geqslant 1$ be an integer. We define $A_{N}(G, \CC)$ to be the universal C*-algebra generated by a unitary $(G, N)$-sudoku matrix $\U$ satisfying the $\CC$-relations. The map
\begin{equation*}
\D : (Q_{h})_{i, j} \mapsto \sum_{k=1}^{N}\sum_{g\in G}(Q_{g})_{ik}\otimes (Q_{hg^{-1}})_{kj}
\end{equation*}
turns $A_{N}(G, \CC)$ into a compact quantum group.
\end{de}

Our aim in this subsection is to prove that this compact quantum group is $G\wr\G_{N}(\CC)$. To do this, we first provide a different presentation of $A_{N}(G, \CC)$.

\begin{lem}\label{lem:presentationsudoku}
Let $B_{N}(G, \CC)$ be the universal C*-algebra generated by a $N\vert G\vert\times N\vert G\vert$ matrix $\V$ satisfying the $\CC$-relations and such that for all $1\leqslant i, j \leqslant N$ and all $g, h, s\in G$,
\begin{equation*}
\V_{i, g}^{j, h} = \V_{i, sg}^{j, sh}.
\end{equation*}
Then, $B_{N}(G, \CC) \simeq A_{N}(G, \CC)$. Moreover, the image of the coproduct on $A_{N}(G, \CC)$ is given by
\begin{equation*}
\D(\V_{i, g}^{j, h}) = \sum_{k=1}^{N}\sum_{s\in G}\V_{i, g}^{k, s}\otimes \V_{k, s}^{j, h}.
\end{equation*}
\end{lem}

\begin{proof}
Let us set $B = B_{N}(G, \CC)$. For $g\in G$ define $Q_{g} \in M_{N}(B)$ by
\begin{equation*}
(Q_g)_{i,j} = \V_{i, g}^{j, e}
\end{equation*}
for $i, j = 1, \dots, N$ and let $\U$ be the associated $(G, N)$-sudoku matrix. This means that for any $g, h \in G$ and $i, j = 1, \dots, N$,
\begin{equation*}
\U_{g,i}^{h,j} = (Q_{hg^{-1}})_{i,j} = \V_{i, hg^{-1}}^{j, e} = \V_{i, g^{-1}}^{j, h^{-1}}.
\end{equation*}
Thus, $\U$ arises from $\V$ via a permutation of the rows and columns, corresponding to identifying two different orderings ($(i, g^{-1})\leftrightarrow (g, i)$) of the indexing set of the canonical orthonormal basis of $\C^{N\times \vert G\vert}$. It is clear that unitarity and the $\CC$-relations are preserved by this operation. Moreover, this operation is invertible since if $\U$ is a $(G, N)$-sudoku matrix, then setting for all $g, h \in G$ and $i, j = 1, \dots, N$
\begin{equation*}
\V_{i,g}^{j,h} = \U_{g^{-1}, i}^{h^{-1}, j} = (Q_{hg^{-1}})_{i,j}
\end{equation*}
yields a unitary matrix $\V$ satisfying the conditions of the statement. The formula for the coproduct then follows from a simple translation of the correspondence between the coefficients of $\U$ and $\V$.
\end{proof}

We are now ready for the main result of this subsection.

\begin{thm}\label{thm:sudoku}
Let $N\geqslant 1$ be an integer, let $\CC$ be a category of partitions and let $G$ be a finite group. Then, $C(G\wr \G_{N}(\CC)) \simeq A_{N}(G, \CC)$.
\end{thm}

\begin{proof}
We will in fact prove that $C(G\wr \G_{N}(\CC))\simeq B_{N}(G, \CC)$ and conclude by Lemma \ref{lem:presentationsudoku}. We therefore have to prove that imposing on a unitary matrix $\V$ the relations making $M_{p}\in \Mor(\V^{\otimes k}, \V^{\otimes l})$ for all $p\in \CC^{G}(k, l)$ is the same as imposing the relations in the statement of Lemma \ref{lem:presentationsudoku}. For $p\in \CC$, let us denote by $\widetilde{p}\in \CC^{G}$ the partition obtained by colouring all the points of $p$ with $e$. Let $q\in \CC^{G}$ and let $p$ be the non-coloured partition underlying $p$. It is clear that $M_{q}$ can be obtained by composing $M_{\widetilde{p}}$ with partitions of the form $M_{\vert_{g}^{e}}$. Thus, the defining relations of $C(G\wr \G_{N}(\CC))$ are exactly the ones coming from $M_{\widetilde{p}}$ and $M_{\vert_{g}^{e}}$ for all $p\in \CC$ and $g\in G$. This enables us to split the proof into two parts.

Recall that $V$ is the vector space with basis $(e_{i}^{g})_{1\leqslant i\leqslant N, g\in G}$. If we identify $V$ with $\C^{N\times \vert G\vert}$, then any partition $p\in \CC(k, l)$ defines a linear map $\widetilde{T}_{p} : V^{\otimes k} \rightarrow V^{\otimes l}$ in the usual way. We claim that this map is equal to $M_{\widetilde{p}}$. In fact,
\begin{equation*}
\widetilde{T}_{p} (e_{i_{1}}^{g_{1}} \otimes \dots \otimes e_{i_{k}}^{g_{k}}) = \sum_{j_{1}, \dots, j_{l} = 1}^{N} \sum_{h_1, \dots h_{l} \in G} \delta_{p}(\bo{i}, \bo{j})\delta_{p}((g_{1}, \dots, g_{k}), (h_{1}, \dots, h_{l})) e_{j_{1}}^{h_{1}}\otimes \dots \otimes e_{j_{l}}^{h_{l}}
\end{equation*}
since under our identification of $V$ both the indices and the colours must be constant on the blocks of $p$ to give a non-zero contribution. But the right-hand side is precisely equal to the vector arising from the action of $M_{\widetilde{p}}$ on $e_{i_{1}}^{g_{1}} \otimes \dots \otimes e_{i_{k}}^{g_{k}}$, proving the claim. Therefore, imposing the relations making all $M_{\widetilde{p}}$'s intertwiners is the same as imposing the $\CC$-relations to the matrix $\V$.

We now further claim that the fact that $M_{\vert_{g}^{e}}$ is an intertwiner of $\V$ means exactly that $\V_{i, s}^{j, h} = \V_{i, gs}^{j, gh}$ for all $h, s \in G$ and $i, j = 1, \dots, N$. Indeed, let us fix $i\in \{1,\ldots, N\}$ and $ h \in G$. On the one hand,
\begin{equation*}
\V\circ (\ii\otimes M_{\vert_{g}^{e}}) (1\otimes e_{i}^{h}) = \V (1\otimes e_{i}^{gh}) = \sum_{j=1}^{N}\sum_{s\in G} \V_{i, gh}^{j, s}\otimes e_{j}^{s}
\end{equation*}
while on the other hand
\begin{equation*}
(\ii\otimes M_{\vert_{g}^{e}})\circ\V  (1\otimes e_{i}^{h}) = \sum_{j=1}^{N}\sum_{s\in G} \V_{i, h}^{j, s}\otimes M_{\vert_{g}^{e}}(e_{j}^{s}) = \sum_{j=1}^{N}\sum_{s\in G} \V_{i, h}^{j, s}\otimes e_{j}^{gs} = \sum_{j=1}^{N}\sum_{s\in G} \V_{i, h}^{j, g^{-1}s}\otimes e_{j}^{s}.
\end{equation*}
Thus, the second claim is proved. To conclude, simply notice that $\D$ is the unique coproduct turning $\V$ into a (fundamental) representation.
\end{proof}

To conclude this subsection, we describe the \emph{abelianization} procedure. Let $\G$ be any compact quantum group and consider the maximal abelian quotient $C(\G)_{ab}$ of $C(\G)$. The coproduct factorizes through this quotient, endowing $C(\G)_{ab}$ with the structure of an "algebra of continuous functions" on a compact quantum group. Because $C(\G)_{ab}$ is commutative by definition, \cite[Thm 1.5]{woronowicz1987compact} implies that it comes from a compact group $G$, which is called the abelianization of $\G$. This operation easily translates at the level of partition wreath products.

\begin{prop}\label{prop:abelianization}
Let $G$ be a finite group, let $\CC$ be a category of non-crossing partitions and let $N\in \N$. Then, the abelianization of $G\wr \G_{N}(\CC)$ is $G\wr \G_{N}(\CC_{ab})$, where $\CC_{ab}$ is the category of partitions generated by $\CC$ and the crossing $\crosspart\in P(2, 2)$. Note that $\G_{N}(\CC_{ab})$ is precisely the abelianization of $\G_{N}(\CC)$.
\end{prop}

\begin{proof}
According to the proof of Theorem \ref{thm:sudoku}, the only relations we add on the matrix $\V$ are the ones making $\widetilde{T}_{\crosspart}$ an intertwiner. But this operator is an intertwiner if and only if all the entries of the matrix commute with one another, hence the result.
\end{proof}

\subsection{Laws of characters}\label{subsec:characters}

Before turning to examples, we will compute the law of the character of the fundamental representation $\V$ in some particular cases. By this we mean analysing the element
\begin{equation*}
\chi_{\V} = (\ii\otimes \Tr)(\V)\in C(G\wr\G_{N}(\CC))
\end{equation*}
as a noncommutative random variable, and more specifically  investigating its distribution (given by its moments) with respect to the Haar state. It was conjectured in \cite[Conj 3.4]{banica2007free} that when $\G_{1}$ and $\G_{2}$ are compact quantum subgroups of $S_{N}^{+}$, then the law of the character of the fundamental representation of $G_{1}\wr_{\ast}\G_{2}$ is the free multiplicative convolution (see for instance \cite[Def 14.1 and Rem 14.5]{nica2006lectures} for the definition) of the laws of the characters of the fundamental representations of $\G_{1}$ and $\G_{2}$. This was proved for free wreath products of arbitrary compact quantum groups by $S_{N}^{+}$ in \cite[Cor 4.4]{lemeux2014free}. As we will see however, the corresponding statement is false when the right-hand side is not a quantum subgroup of $S_{N}^{+}$. Because of our description of the intertwiner spaces, the computation of the moments reduces to a counting problem on partitions. To lighten the formul\ae{}, we will write $\CC_{n}$ for $\CC(0, n)$.

\begin{prop}\label{prop:moments}
Let $\CC$ be a category of partitions and let $N\in \N$ be such that for all $k, l \in \N$ the collection of linear maps $\{T_{p} : p \in \CC(k,l)\}$ is linearly independent. Then for any finite group $G$ the moments $(m_{n}(\chi))_{n\in\N}$ of the character $\chi_{\V}$ of the fundamental representation $\V$ of $G\wr\G_{N}(\CC)$ with respect to the Haar state are given by the following formula:
\begin{equation*}
m_{n}(\chi_{\V}) = \sum_{p \in \CC_{n}} \prod_{b \subset p} \vert b\vert^{\vert G\vert-1}.
\end{equation*}
\end{prop}

\begin{proof}
It follows from the results of S.L. Woronowicz in \cite{woronowicz1987compact} that the moments in question are given by the formula
\begin{equation}\label{eq:fixedpoints}
m_{n}(\chi_{\V}) = \dim(\Mor(\varepsilon, \V^{\otimes n}))
\end{equation}
where $\varepsilon$ denotes the trivial representation. Let us denote by $\CC^{G}_{n}(e)$ the set of partitions in $\CC^{G}_{n}$ for which the leftmost colour in every block of $p$ is the neutral element $e$. For any $p\in \CC^{G}_{n}$, there is a unique $q\in \CC^{G}_{n}(e)$ such that $M_{p} = M_{q}$, so that the space on the right-hand side of Equation \eqref{eq:fixedpoints} is spanned by $\{M_{p} : p \in \CC^{G}_{n}(e)\}$. We claim that these operators are linearly independent. To prove it, let us consider a vanishing linear combination
\begin{equation*}
X = \sum_{p\in \CC_{n}(e)}\lambda_{p}M_{p} = 0.
\end{equation*}
Using the definition of $M_{p}$, we can expand this sum as
\begin{equation*}
X = \sum_{p\in \CC_{n}(e)}\lambda_{p}\sum_{\bo{x}\leqslant p}T_{\bo{x}.p}.
\end{equation*}
We will now reorder the terms of this sum according to the colouring of the arguments of the $T$'s. Note that for any tuple $\bo{g} = (g_{1}, \dots, g_{n})$ of elements of $G$ and for each $p\in \CC_{n}(e)$, there is at most one $\bo{x}\leqslant p$ such that $\bo{x}.p $ has colouring $\bo{g}$. Let us write $p\sim \bo{g}$ if such an $\bo{x}$ exists and let us denote by $\overline{p}$ the uncoloured partition underlying $p$. Then,
\begin{equation*}
X = \sum_{\bo{g}}\left(\sum_{\underset{p\in \CC_{n}(e)}{p\sim \bo{g}}}\lambda_{p}T_{\overline{p}(\bo{g})}\right) = \sum_{\bo{g}} X(\bo{g}).
\end{equation*}
The operators $X(\bo{g})$ for different tuples $\bo{g}$ act on orthogonal subspaces so that $X = 0$ implies $X(\bo{g}) = 0$ for each $\bo{g}$. When considering a fixed $X(\bo{g})$, the colouring is fixed so that we can use the linear independence assumption to conclude that $\lambda_{p} = 0$ for all $p\sim\bo{g}$, proving the linear independence. We can now compute
\begin{equation*}
m_{n}(\chi) = \left\vert\CC^{G}_{n}(e)\right\vert = \sum_{q \in \CC_{n}} \left\vert\{p\in\CC^{G}_{n}(e) : \bar{p} = q\}\right\vert = \sum_{q \in \CC_{n}} \prod_{b \subset q} \vert G\vert^{\vert b\vert-1},
\end{equation*}
where the last equality is obtained by counting all the possible colourings.
\end{proof}

By Proposition \ref{prop:linearindependence}, the assumptions of Proposition \ref{prop:moments} are satisfied if $\CC$ is a category of non-crossing partitions and $N \geqslant 4$. Let us denote by $\chi_{G}\in C(G)$ the character of the left regular representation of $G$ and by $\mu_{G}$ its law with respect to the Haar measure. This yields a probability measure depending only on the cardinality of $G$. Since it will be convenient to express the results in terms of $\mu_{G}$, we recall the explicit formula (see for instance \cite[Prop 4.1]{banica2007free}):
\begin{equation*}
\mu_{G} = \left(1 - \frac{1}{\vert G\vert}\right)\delta_{0} + \frac{1}{\vert G\vert}\delta_{\vert G\vert}.
\end{equation*}
We start with the case of quantum permutations, i.e. $\CC = NC$.

\begin{cor}
Assume that $\CC = NC$, let $N \geqslant 4$ and let $G$ be a finite group. Then the moments $(m_{n}(\chi_{\V}))_{n\in \N}$ of $\chi_{\V}$ are given by
\begin{equation*}
m_{n}(\chi_{\V}) = \sum_{p \in NC(n)} \prod_{b \subset p} \vert G\vert^{\vert b\vert-1}.
\end{equation*}
\end{cor}

Using the formalism of free cumulants (see for instance \cite[Def 11.3]{nica2006lectures}), we see that the $n$-th free cumulant of the distribution of the character $\chi_{\V}$ is equal to $\vert G\vert^{n-1}$ for every $n \in \N$. In other words, the law of $\chi_{\V}$ is a free compound Poisson law with initial law $\mu_{G}$ and parameter $1$. This recovers a particular case of \cite[Prop 4.3]{lemeux2014free}. In particular, the law of $\chi_{\V}$ is the free multiplicative convolution of $\mu_{G}$ and the law of the fundamental representation of $S_{N}^{+}$ (which is a free Poisson law with parameter $1$). The situation is quite different if we consider only non-crossing pair partitions, i.e. the case of $O_{N}^{+}$.

\begin{cor}
Assume that $\CC = NC_{2}$, let $N \geqslant 4$ and let $G$ be a finite group. Then the even moments $(m_{2n}(\chi_{\V}))_{n\in \N}$ of $\chi_{\V}$ are given by
\begin{equation*}
m_{2n}(\chi_{\V}) = \sum_{p \in NC_{2}(2n)} \vert G\vert^{n}
\end{equation*}
and the odd moments vanish.
\end{cor}

Again this is easily interpreted in terms of cumulants: $k_{2}(\chi_{\V}) = \vert G\vert$ and all other cumulants vanish. Thus, $\chi_{\V}$ is a centred semi-circular variable with variance $\vert G\vert$. By \cite[Ex 14.21]{nica2006lectures} (see also \cite[Thm 4.4]{banica2007free}), the free multiplicative convolution of $\mu_{G}$ and any probability distribution $\mu$ is given by
\begin{equation*}
\mu_{G}\boxtimes\mu = \left(1 - \frac{1}{\vert G\vert}\right)\delta_{0} + \frac{1}{\vert G\vert}\mu^{\boxplus \vert G\vert}.
\end{equation*}
The free additive convolution of semi-circular laws is the semicircular law whose variance is the sum of the variances. Thus, the law of $\chi_{\V}$ is \emph{not} the free multiplicative convolution of $\mu_{G}$ and the law of the fundamental representation of $O_{N}^{+}$ (which is a centred semicircular law with variance $1$). Our last example will be the hyperoctahedral case, were $\CC = NC_{ev}$ is the category of all non-crossing partitions with blocks of even size.

\begin{cor}
Assume that $\CC = NC_{ev}$, let $N \geqslant 4$ and let $G$ be a finite group. Then, the even moments $(m_{2n}(\chi))_{n\in \N}$ of $\chi_{\V}$ are given by
\begin{equation*}
m_{2n}(\chi_{\V}) = \sum_{p \in NC_{ev}(2n)}\prod_{b\subset p} \vert G\vert^{\vert b\vert - 1}
\end{equation*}
and the odd moments vanish.
\end{cor}

As in the quantum permutation case, this is a compound free Poisson law with parameter $1$. The initial law $\mu$ has vanishing odd moments and even moments equal to $\vert G\vert^{n-1}$. The only probability measure with these moments is
\begin{equation*}
\mu = \frac{1}{2\vert G\vert}\delta_{-\vert G\vert} + \left(1-\frac{1}{\vert G\vert}\right)\delta_{0} + \frac{1}{2\vert G\vert}\delta_{\vert G\vert}.
\end{equation*}
Note that this measure can be obtained as the \emph{classical} multiplicative convolution of $\mu_{G}$ by the Bernoulli measure
\begin{equation*}
\nu = \frac{1}{2}\delta_{-1} + \frac{1}{2}\delta_{1}.
\end{equation*}

\subsection{Two examples}

In some cases, it is possible to relate our construction to usual (free) wreath products by the (free) permutation groups. This can be seen using the sudoku picture explained above.

\subsubsection{The hyperoctahedral case}\label{subsec:hyperoctahedralgeneral}

Our first example is the partition wreath product with the hyperoctahedral quantum group $H_{N}^{+}$. This corresponds to the category of partitions $\CC = NC_{ev}$ consisting of all non-crossing partitions with blocks of even size.

\begin{thm}\label{thm:hyperoctahedralgeneral}
Let $N \geqslant 1$ be an integer and let $G$ be a finite group. Then, $G\wr_{\ast} H_{N}^{+}$ is isomorphic to $(G\times \Z_{2})\wr S_{N}^{+}$.
\end{thm}

\begin{proof}
Let $A_{N}(G, NC_{2})$ be the C*-algebra of Definition \ref{de:sudoku}. Recall that it is generated by an $N\vert G\vert \times N\vert G\vert$ unitary matrix $\U$ satisfying the $NC_{ev}$-relations (i.e. each entry is a self-adjoint partial isometry) and such that it has the form $\U = (Q_{h^{-1}g})_{g,h\in G}$ for some $N\times N$ matrices of operators $(Q_{r})_{r\in G}$. We are going to use a "doubling trick" to produce a bigger sudoku matrix. Let us set, for every $g, h\in G$ and $i, j= 1, \dots, N$,
\begin{equation*}
(U^{i,j}_{g,h})^{+} = \frac{1}{2}\left(U^{i, j}_{g, h} + (U^{i, j}_{g, h})^{2} \right) \text{ and } (U^{i, j}_{g, h})^{-} = \frac{1}{2}\left(U^{i, j}_{g, h} - (U^{i, j}_{g, h})^{2} \right).
\end{equation*}
We will denote by $\U^{+}$ the matrix obtained by replacing the coefficients by their "$+$" part and by $\U^{-}$ the one obtained with the "$-$" part. It follows from the proof of point $(2)$ of \cite[Thm 3.4]{banica2009fusion} (with $s = 2$) that the matrix
\begin{equation*}
\W = \left(\begin{array}{cc}
\U^{+} & \U^{-} \\
\U^{-} & \U^{+} \\
\end{array}\right)
\end{equation*}
is a magic unitary of size $2N\vert G\vert = N\vert G \times\Z_{2}\vert$ and that the unique coproduct turning $\W$ into a representation is compatible with that of $A_{N}(G, \CC)$. Consider now for every $r\in G$ and for $t = 0, 1$ the following $N$ by $N$ matrices:
\begin{equation*}
P_{r,t} = \left\{\begin{array}{ccc}
Q_{r}^{+} & \text{if} & t=0 \\
Q_{r}^{-} & \text{if} & t=1 \\
\end{array}\right..
\end{equation*}
Then, $\W$ is equal to the sudoku matrix $(P_{(h,t)^{-1}(g,u)})_{(g,u),(h,t) \in G\times \Z_{2}}$. The procedure transforming $\U$ into $\W$ is clearly reversible, hence the C*-algebra $A_{N}(G, NC_{ev})$ is isomorphic to the C*-algebra $A_{N}(G\times \Z_{2}, NC)$ and the isomorphism preserves the respective coproducts. Applying Theorem \ref{thm:sudoku} now yields the result.
\end{proof}

\begin{rem}
We will prove in Theorem \ref{thm:freewreath} that $G\wr S_{N}^{+}$ is in fact isomorphic to $G\wr_{\ast}S_{N}^{+}$ for any finite group $G$. Thus, Theorem \ref{thm:hyperoctahedralgeneral} really gives an explicit description of the partition wreath product $G\wr H_{N}^{+}$. Moreover, it shows that there is no general "associativity" property of the partition wreath product since
\begin{equation*}
G\wr H_{N}^{+} = G\wr (\Z_{2}\wr S_{N}^{+}) = (G\times \Z_{2})\wr S_{N}^{+} \neq (G\wr\Z_{2})\wr S_{N}^{+}.
\end{equation*}
\end{rem}

Note that the isomorphism of Lemma \ref{lem:presentationsudoku} between $A_{N}(G, \CC)$ and $B_{N}(G, \CC)$ does not change the character of the defining fundamental representation, whereas the procedure used in the above proof does. With the notations used above, the character of $\U$ in $G\wr_{\ast} H_{N}^{+} $ has the form
\begin{equation*}
\chi = \vert G\vert\sum_{i=1}^{N} Q_{e}^{i, i}
\end{equation*}
whereas the character of $\W$ in $(G\times \Z_{2}) \wr_{\ast} S_{N}^{+}$ is
\begin{equation*}
\tilde{\chi} = 2\vert G\vert\sum_{i=1}^{N} (Q_{e}^{(i, i)})^{+} = \vert G\vert\sum_{i=1}^{N} \left(Q_{e}^{(i, i)} + (Q_{e}^{(i, i)})^{2}\right).
\end{equation*}
This is the reason why the initial law of the character of $\V$ computed in Subsection \ref{subsec:characters} is not $\mu_{G\times\Z_{2}}$. Using abelianization, we can also describe the classical case.

\begin{cor}\label{cor:abelianizedhyperoctahedral}
Let $N\geqslant 1$ be an integer and let $G$ be a finite group. Then, $G\wr H_{N}$ is isomorphic to $(G\times \Z_{2})\wr S_{N}$.
\end{cor}

\begin{proof}
This follows from Theorem \ref{thm:hyperoctahedralgeneral} and Proposition \ref{prop:abelianization}.
\end{proof}

\subsubsection{The case of free permutations}\label{subsec:permutation}

Recall that our justification for the notation $G\wr\G_{N}(\CC)$ was that when $\G_{N}(\CC)$ is the permutation group $S_{N}$, we recover the usual wreath product $G\wr S_{N}$. To conclude this section, we will now prove that if $\G_{N}(\CC)$ is the quantum permutation group $S_{N}^{+}$, then our construction yields the free wreath product $G\wr_{\ast} S_{N}^{+}$.

\begin{thm}\label{thm:freewreath}
Let $N\geqslant 1$ be an integer and let $G$ be a finite group. The quantum group $G \wr \G_{N}(NC)$ is isomorphic to the free wreath product $G\wr_{\ast}S_{N}^{+}$.
\end{thm}

\begin{proof}
The proof is a generalisation of the one of \cite[Thm 3.4]{banica2009fusion}. Recall that the $NC$-relations are the so-called \emph{magic relations}, i.e. the matrix $\U$ is unitary and all its coefficients are orthogonal projections. Let us first consider the free wreath product $G\wr_{\ast} S_{N}^{+}$, where the fundamental representation $u\in M_{\vert G\vert}(C(G))$ of $C(G)$ is given by $u_{g, h}(t) = \delta(h, tg)$. Notice that $u$ is a magic unitary matrix and that $u_{g, h} = u_{e, hg^{-1}}$. Considering the $i$-th copy $u^{(i)}$ of this fundamental representation, we set
\begin{equation*}
(Q_{g})_{ij} = u_{e, g}^{(i)}v_{ij}.
\end{equation*}
Then, $\U_{g, i}^{h, j} = u_{g, h}^{(i)}v_{ij} = (Q_{hg^{-1}})_{ij}$ is a $(G, N)$-sudoku matrix. Because $u^{(i)}$ commutes with $v_{ij}$ for all $1\leqslant j\leqslant N$ and both $u$ and $(v_{ij})_{ij}$ are magic unitaries, $\U$ is a magic unitary matrix. Thus, there is a surjective $*$-homomorphism $\Phi : A_{N}(G, NC)\longrightarrow C(G\wr_{\ast} S_{N}^{+})$ given by $\Phi(\U_{g, i}^{h, j}) = u_{g, h}^{(i)}v_{ij}$.

Consider now the generating $(G, N)$-sudoku matrix $\U = (Q_{h^{-1}g})_{g, h\in G}$ of $A_{N}(G, NC)$ and set, for $1\leqslant i, j\leqslant N$ and $g, h\in G$,
\begin{equation*}
u_{g, h}^{(i)} = \sum_{k=1}^{N}(Q_{hg^{-1}})_{ik} \text{ and } v_{ij} = \sum_{g\in G}(Q_{g})_{ij}.
\end{equation*}
Because $\U$ is magic unitary, $(v_{ij})_{ij}$ is a magic unitary matrix. Using the fact that $u_{g, h}^{(i)} = u_{e, hg^{-1}}^{(i)}$ and that $\U$ is magic unitary, we see that for a fixed $i$, the elements $\{u_{g, h}^{(i)}: g, h \in G\}$ are commuting pairwise orthogonal projections summing up to $1$. This means that the algebra generated by these elements is of the form $C(H)$ for some subgroup $H$ of $S_{\vert G\vert}$. Moreover, the equality $u_{g, h}^{(i)} = u_{e, hg^{-1}}^{(i)}$ implies that in fact $H = G$. Thus, there is a unique surjective $*$-homomorphism $\Psi : C(G)^{\ast N}\ast C(S_{N}^{+}) \longrightarrow A_{N}(G, NC)$ such that the fundamental representation of $S_{N}^{+}$ is sent to $(v_{ij})_{ij}$ and the fundamental representation of the $i$-th copy of $C(G)$ is sent to $(u^{(i)}_{g, h})_{g, h}$. Using the fact that the $(G, N)$-sudoku matrix $\U$ is magic, we have for all $i,j=1,\ldots,N, g,h \in G$
\begin{eqnarray*}
u_{g, h}^{(i)}v_{ij} & = & \left(\sum_{k=1}^{N}(Q_{hg^{-1}})_{ik}\right)\left(\sum_{t\in G}(Q_{t})_{ij}\right) \\
& = & \sum_{k=1}^{N}\sum_{t\in G}\delta(t, hg^{-1})\delta(k, j)(Q_{hg^{-1}})_{ik}(Q_{t})_{ij} \\
& = & (Q_{hg^{-1}})_{ij} \\
& = & \left(\sum_{t\in G}(Q_{t})_{ij}\right)\left(\sum_{k=1}^{N}(Q_{hg^{-1}})_{ik}\right) \\
& = & v_{ij}u_{g, h}^{(i)}
\end{eqnarray*}
Thus, $\Psi$ factorizes through a surjective $*$-homomorphism $\widetilde{\Psi} : C(G\wr_{\ast}S_{N}^{+}) \longrightarrow A_{N}(G, NC)$. Checking that $\Phi$ and $\widetilde{\Psi}$ are inverse to each other is now routine calculation.
\end{proof}

\begin{rem}
Note that here we are considering the finite group $G$ as a \emph{compact} group. The compact quantum group $G\wr S_{N}^{+}$ is therefore not the same as the free wreath product $\widehat{G}\wr_{*} S_{N}^{+}$ studied by F. Lemeux in \cite{lemeux2013fusion}. It however fits into the more general setting of free wreath products of compact quantum groups by $S_{N}^{+}$ treated by F. Lemeux and P. Tarrago in \cite{lemeux2014free}. By restricting to finite groups, our approach gives a different description of the situation.
\end{rem}

\begin{rem}
Note that combining Theorem \ref{thm:freewreath} with the abelianization result of Proposition \ref{prop:abelianization}, we recover the classical case of Corollary \ref{cor:classicalpermutation}.
\end{rem}

\begin{rem}
Part of the proof of Theorem \ref{thm:freewreath} applies in full generality, i.e.\ not only when $\CC= NC$. Indeed, one can easily check that the formula $v_{ij} = \sum_{g\in G} (Q_{g})_{ij}$ always defines a copy of the fundamental representation of $\G_{N}(\CC)$ inside $C(G\wr \G_{N}(\CC))$. Moreover,
\begin{equation*}
u^{(i)}_{g, h} = \sum_{k=1}^{N} (Q_{hg^{-1}})_{ik}
\end{equation*}
always defines a unitary matrix $(u^{(i)}_{g, h})_{g,h \in G}$ with coefficients in $A(G, \CC)$. However, this unitary matrix will not be magic if $\U$ is not assumed to be magic, hence it is \emph{not} a copy of the regular representation of $G$. As some examples of Subsection \ref{subsec:examples} will show, one cannot in general recover canonical copies of the algebra $C(G)$ inside $C(G\wr \G_{N}(\CC))$.
\end{rem}

To end this section, let us point out the two main differences between our partition wreath product and the free wreath product:
\begin{enumerate}
\item Our construction is a generalisation of the wreath product in the sense that when $\G_{N}(\CC)$ is a classical group, the construction yields a classical group. On the contrary, the free wreath product $G\wr_{\ast} S_{N}$ is \emph{not} isomorphic to the wreath product $G\wr S_{N}$ as soon as $G$ is non-trivial.
\item The definition of the free wreath product only yields a compact quantum group when the right-hand side is a quantum subgroup of $S_{N}^{+}$. Our construction works for many compact quantum groups which are not quantum subgroups of $S_{N}^{+}$ but does not work for arbitrary quantum subgroups of $S_{N}^{+}$.
\end{enumerate}

\section{Partition wreath products of abelian groups as partition quantum groups}\label{sec:abelian}

In this section, we will give yet another picture of the partition wreath product $G\wr\G_{N}(\CC)$ when $G$ is abelian. More precisely, we will show that $G\wr\G_{N}(\CC)$ is then isomorphic to a \emph{partition quantum group}. This means that its representation theory can be described using the results of \cite{freslon2013representation}, as explained in \cite{freslon2014partition}. The drawback of this approach is of course the restriction to abelian groups. We will explain in the end of the section why this restriction seems necessary.

\subsection{Basic representations}

The first step is to decompose the fundamental representation $\V$ into a sum of irreducible representations. The latter will then be our building blocks for the study of $G\wr \G_{N}(\CC)$. To do this, we have to find minimal projections in the algebra $\Mor(\V, \V)$. Recall that $G$ is now assumed to be abelian and set, for a character $\chi\in \h{G}$,
\begin{equation*}
P_{\chi} = \frac{1}{\vert G\vert}\sum_{g\in G}\chi(g)M_{\vert^{g}_{e}}.
\end{equation*}
We denote the trivial character in $\h{G}$ by $\varepsilon$.

\begin{lem}\label{lem:basicprojections}
The operators $P_{\chi}$ are pairwise orthogonal self-adjoint projections summing to $1$. Moreover, for every non-trivial $\chi \in \h{G}$ the projection $P_{\chi}$ is minimal in $\Mor(\V, \V)$ and $P_{\varepsilon}$ is either minimal in $\Mor(\V, \V)$ (when $\CC(1,1)$ does not contain the two-block partition) or decomposes into two minimal projections (if $\CC(1,1)$  contains the two-block partition).
\end{lem}

\begin{proof}
Note that for all $g,h \in G$ we have the equality
\begin{equation*}
M_{\vert^{g}_{e}}M_{\vert^{h}_{e}} = M_{\vert^{gh}_{e}}.
\end{equation*}
The fact that the operators are self-adjoint projections is straightforward, pairwise orthogonality follows from the orthogonality relations for characters and the fact that the projections sum to $1$ is a consequence of the equality $\sum_{\chi \in \h{G}} \chi (g)= \vert G\vert\delta_{e,g}$ for all $g \in G$. Moreover, $\Mor(\V, \V)$ is spanned by the operators $M_{\vert^{g}_{e}}$ (and possibly also the one corresponding to the partition $p$ in $\CC(1,1)$ consisting of the two one-point sets, if the corresponding partition is in the category), hence has dimension at most $\vert G\vert$, or $\vert G \vert +1$ (in the second case). Thus in the case where $p$ is not in $\CC$ all the projections $P_\chi$ are minimal, and when $p \in \CC$ it is easy to see that $P_{\varepsilon}$ contains non-trivially the projection corresponding to $p$.
\end{proof}

We can now decompose $\V$. Set, for $\chi\in \h{G}$, $u^{\chi} = (\ii\otimes P_{\chi})(\V)$. By Lemma \ref{lem:basicprojections}, the representations $u^{\chi}$ are pairwise non-equivalent irreducible representations (if the two-block partition belongs to $\CC(1,1)$ -- otherwise all $u^{\chi}$ for $\chi\neq \varepsilon$ are irreducible and the representation $u^{\varepsilon}$ decomposes into two irreducible representations, one of which is trivial), whose direct sum is $\V$. Let us give a name to these representations.

\begin{de}
The representations $u^{\chi}$ are called the \emph{basic representations} of $G\wr\G_{N}(\CC)$.
\end{de}

Because $\V$ can be completely decomposed using basic representations, it is enough to study tensor products of these basic representations to understand the full representation theory of $G\wr \G_{N}(\CC)$. Let $\chi_{1}, \dots, \chi_{k}$ and $\rho_{1}, \dots, \rho_{l}$ be characters of $G$. Then,
\begin{equation*}
\Mor(u^{\chi_{1}}\otimes \dots \otimes u^{\chi_{k}}, u^{\rho_{1}}\otimes\dots \otimes u^{\rho_{l}}) = (P_{\chi_{1}}\otimes \dots\otimes P_{\chi_{k}})\Mor(\V^{\otimes k}, \V^{\otimes l})(P_{\rho_{1}}\otimes \dots\otimes P_{\rho_{l}}).
\end{equation*}
To give a better description of this space, we will first define some specific operators in $\Mor(\V^{\otimes k}, \V^{\otimes l})$.

\subsection{Averaging morphisms by characters}

We now introduce operators in $\Mor(\V^{\otimes k}, \V^{\otimes l})$ obtained by taking linear combinations whose coefficients are given by characters of $G$. This definition involves many parameters and is therefore complicated to write down. To keep things tractable, we first introduce some additional notations. Recall that to a non-coloured partition $p\in P(k, l)$ and two tuples $\bo{g} = (g_{1}, \dots, g_{k})$, $\bo{h} = (h_{1}, \dots, h_{l})$ of elements of $G$ we can associate a coloured partition $p(\bo{g}, \bo{h})$ by colouring (from left to right) the lower points of $p$ by $\bo{g}$ and its upper points (also from left to right) by $\bo{h}$. If $\bo{\chi} = (\chi_{1}, \dots, \chi_{k})$ is a tuple of characters of $G$, we set, for a tuple $\bo{g} = (g_{1}, \dots, g_{k})$ of elements of $G$, $\bo{\chi}(\bo{g}) = \chi_{1}(g_{1})\dots \chi_{k}(g_{k})$.

\begin{de}
Let $p\in P(k, l)$ be a partition and let $\bo{\chi} = (\chi_{1}, \dots, \chi_{k})$ and $\bo{\rho} = (\rho_{1}, \dots, \rho_{l})$ be tuples of characters of $G$. The associated averaged operator is
\begin{equation*}
F_{p}(\bo{\chi}, \bo{\rho}) = \sum_{\bo{g} = (g_{1}, \dots, g_{k})}\sum_{\bo{h} = (h_{1}, \dots, h_{l})}\bo{\chi}(\bo{g})\overline{\bo{\rho}}(\bo{h})M_{p(\bo{g}, \bo{h})}.
\end{equation*}
\end{de}

The first important fact about the map $F_{p}$ is that it is $0$ unless the tuples of characters satisfy a certain compatibility condition which will be called \emph{$p$-admissibility}. To explain this condition, let us first introduce some notations.  If $b$ is a block of a partition $p\in P(k, l)$ and if $\bo{a} = (a_{1}, \dots, a_{k})$ is any tuple, then we denote by $\bo{a}^{b}$ the restriction of the tuple $\bo{a}$ to the indices which belong to the upper row of $b$. Similarly, if $\bo{a} = (a_{1}, \dots, a_{l})$ then we denote by $\bo{a}_{b}$ its restriction to the indices which belong to the lower row of $b$. If $\bo{\chi} = (\chi_{1}, \dots, \chi_{k})$ is a tuple of characters of $G$, we set $\prod\bo{\chi} = \chi_{1}\dots \chi_{k}$.

\begin{de}
Let $p\in P(k, l)$ and consider two tuples $\bo{\chi} = (\chi_{1}, \dots, \chi_{k})$ and $\bo{\rho} = (\rho_{1}, \dots, \rho_{l})$ of characters of $G$. The pair $(\bo{\chi}, \bo{\rho})$ is said to be $p$-admissible if for any block $b$ of $p$, $\prod\bo{\chi}_{b} = \prod\bo{\rho}^{b}$.
\end{de}

\begin{lem}\label{lem:admissible}
Let $p\in P(k, l)$ and let $\bo{\chi} = (\chi_{1}, \dots, \chi_{k})$ and $\bo{\rho} = (\rho_{1}, \dots, \rho_{l})$ be two tuples of characters of $G$. If $(\bo{\chi}, \bo{\rho})$ is not $p$-admissible, then $F_{p}(\bo{\chi}, \bo{\rho}) = 0$.
\end{lem}

\begin{proof}
If $\bo{x} = (x_{1}, \dots, x_{b(p)})$, then $M_{\bo{x}.p(\bo{g}, \bo{h})} = M_{p(\bo{g}, \bo{h})}$. Thus,
\begin{equation*}
F_{p}(\bo{\chi}, \bo{\rho}) = \sum_{\bo{g}, \bo{h}}\bo{\chi}(\bo{g})\overline{\bo{\rho}}(\bo{h})\vert G\vert^{-b(p)}\sum_{\bo{x}\leqslant p}M_{\bo{x}.p(\bo{g}, \bo{h})}
\end{equation*}
where as in Lemma \ref{lem:compositionaveraged} the notation $\bo{x}\leqslant p$ means that we are considering a tuple $\bo{x} = (x_{1}, \dots, x_{b(p)})$. Making the changes of variables $\bo{g}' = \bo{x}^{d}\bo{g}$ and $\bo{h}' = \bo{x}^{u}\bo{h}$ yields
\begin{equation*}
F_{p}(\bo{\chi}, \bo{\rho}) = \vert G\vert^{-b(p)}\sum_{\bo{g}', \bo{h}'}\sum_{\bo{x}\leqslant p}\bo{\chi}((\bo{x}^{d})^{-1}\bo{g}')\overline{\bo{\rho}}((\bo{x}^{u})^{-1}\bo{h}')M_{p(\bo{g}', \bo{h}')}.
\end{equation*}
We can rename the variables to remove the primes and make the change of variables $\bo{x}\mapsto \bo{x}^{-1}$ to lighten the expression, thus obtaining
\begin{equation*}
F_{p}(\bo{\chi}, \bo{\rho}) = \vert G\vert^{-b(p)}\sum_{\bo{g}, \bo{h}}\left(
\sum_{\bo{x}\leqslant p}\bo{\chi}(\bo{x}^{d})\overline{\bo{\rho}}(\bo{x}^{u})\right)\bo{\chi}(\bo{g})\overline{\bo{\rho}}(\bo{h})M_{p(\bo{g}, \bo{h})}.
\end{equation*}
Let us consider the term in parenthesis. We are summing over tuples $\bo{x}$ which are indexed by the blocks of $p$. Thus, the sum can be written as a product over the blocks:
\begin{equation*}
\prod_{b\subset p}\left(\sum_{x\in G}(\prod\bo{\chi}_{b})(x)(\prod\overline{\bo{\rho}}^{b})(x)\right).
\end{equation*}
Because of the orthogonality relations for characters, this product vanishes as soon as there is a block $b$ for which $\prod\bo{\chi}_{b} \neq \prod\bo{\rho}^{b}$, hence the result.
\end{proof}

The main feature of the maps $F_{p}$ is their interplay with the category operations and in particular with the composition.

\begin{prop}\label{prop:compositionfp}
Let $p\in P(l, m)$, $q\in P(k, l)$ and fix tuples $(\bo{\chi}, \bo{\rho})$ and $(\bo{\chi}', \bo{\rho}')$. Then,
\begin{equation*}
F_{p}(\bo{\chi}, \bo{\rho})F_{q}(\bo{\chi}', \bo{\rho}') = \vert G\vert^{b(p)+b(q)-b(pq)+l}N^{\rl(p, q)}\delta(\bo{\rho}, \bo{\chi}')F_{pq}(\bo{\chi}, \bo{\rho}').
\end{equation*}
\end{prop}

\begin{proof}
First note that because of Lemma \ref{lem:admissible}, we can assume that $(\bo{\chi}, \bo{\rho})$ and $(\bo{\chi}', \bo{\rho}')$ are $p$-admissible since otherwise both sides of the equality vanish. Let us also make a simple remark: in Lemma \ref{lem:compositionaveraged} the resulting averaged operator does not depend on the choice of $\bo{x}$ and $\bo{y}$. Thus, we can sum over all possible such colourings of the blocks. If the composition is non-zero, then there are exactly $\vert G\vert^{\rl(p, q) + b(pq)}$ colourings giving the same averaged operator. Thus, Lemma \ref{lem:compositionaveraged} can be restated as follows:
\begin{equation*}
M_{p(\bo{g}, \bo{h})}M_{q(\bo{s}, \bo{t})} = \vert G\vert^{-b(pq)}N^{\rl(p, q)}\sum_{\bo{x}\leqslant p}\sum_{\bo{y}\leqslant q}\delta(\bo{c}_{u}(\bo{x}.p(\bo{g}, \bo{h})), \bo{c}_{d}(\bo{y}.q(\bo{s}, \bo{t})))M_{pq(\bo{c}_{d}(\bo{x}.p(\bo{g}, \bo{h})), \bo{c}_{u}(\bo{y}.q(\bo{s}, \bo{t})))}.
\end{equation*}
Using this, the product $F_{p}(\bo{\chi}, \bo{\rho})F_{q}(\bo{\chi}', \bo{\rho}')$ becomes
\begin{equation*}
\vert G\vert^{-b(pq)}N^{\rl(p, q)}\sum_{\bo{g}, \bo{h}, \bo{s}, \bo{t}}\sum_{\bo{x}\leqslant p, \bo{y}\leqslant q}\bo{\chi}(\bo{g})\overline{\bo{\rho}}(\bo{h})\bo{\chi}'(\bo{s})\overline{\bo{\rho}}'(\bo{t})\delta(\bo{c}_{u}(\bo{x}.p(\bo{g}, \bo{h})), \bo{c}_{d}(\bo{y}.q(\bo{s}, \bo{t})))M_{pq(\bo{c}_{d}(\bo{x}.p(\bo{g}, \bo{h})), \bo{c}_{u}(\bo{y}.q(\bo{s}, \bo{t})))}.
\end{equation*}
We make the following changes of variables:
\begin{equation*}
\bo{g}' = \bo{x}^{d}.\bo{g}, \bo{h}' = \bo{x}^{u}.\bo{h}, \bo{s}' = \bo{y}^{d}.\bo{s}, \bo{t}' = \bo{y}^{u}.\bo{t}
\end{equation*}
to get
\begin{equation*}
\vert G\vert^{-b(pq)}N^{\rl(p, q)}\sum_{\bo{g}', \bo{h}', \bo{s}', \bo{t}'}\sum_{\bo{x}\leqslant p, \bo{y}\leqslant q}\delta(\bo{h}', \bo{s}')\bo{\chi}((\bo{x}^{d})^{-1}.\bo{g}')\overline{\bo{\rho}}((\bo{x}^{u})^{-1}.\bo{h}')\bo{\chi}'((\bo{y}^{d})^{-1}.\bo{s}')\overline{\bo{\rho}}'((\bo{y}^{u})^{-1}.\bo{t}')M_{pq(\bo{g}', \bo{t}')}.
\end{equation*}
Because of the $\delta$-symbol, one tuple in the sum disappears. Let us rename the variables without primes and make the changes of variables $\bo{x}\mapsto \bo{x}^{-1}$ and $\bo{y}\mapsto \bo{y}^{-1}$. Reordering the terms yields
\begin{equation*}
\vert G\vert^{-b(pq)}N^{\rl(p, q)}\sum_{\bo{g}, \bo{t}}\sum_{\bo{x}\leqslant p, \bo{y}\leqslant q}\left(\sum_{\bo{h}}\overline{\bo{\rho}}(\bo{x}^{u}.\bo{h})\bo{\chi}'(\bo{y}^{u}.\bo{h})\right)\bo{\chi}(\bo{x}^{d}.\bo{g})\overline{\bo{\rho}}'(\bo{y}^{u}.\bo{t})M_{pq(\bo{g}, \bo{t})}.
\end{equation*}
Using the orthogonality relations for characters, the term in parenthesis can be computed:
\begin{equation*}
\sum_{\bo{h}}\overline{\bo{\rho}}(\bo{x}^{u}.\bo{h})\bo{\chi}'(\bo{y}^{d}.\bo{h}) = \vert G\vert^{l}\delta(\bo{\chi}', \bo{\rho})\overline{\bo{\rho}}(\bo{x}^{u})\bo{\chi}'(\bo{y}^{d}).
\end{equation*}
We therefore now have to compute
\begin{eqnarray*}
X(\bo{g}, \bo{t}) & = & \sum_{\bo{x}\leqslant p, \bo{y}\leqslant q}\bo{\chi}(\bo{x}^{d}.\bo{g})\overline{\bo{\rho}}(\bo{x}^{u})\bo{\chi}'(\bo{y}^{d})\overline{\bo{\rho}}'(\bo{y}^{u}.\bo{t}) \\
& = & \left(\sum_{\bo{x}\leqslant p}
\bo{\chi}(\bo{x}^{d})\overline{\bo{\rho}}(\bo{x}^{u})\right)
\left(\sum_{\bo{y}\leqslant q}\bo{\chi}'(\bo{y}^{d})\overline{\bo{\rho}}'(\bo{y}^{u})\right)\bo{\chi}(\bo{g})\overline{\bo{\rho}}(\bo{t})
\end{eqnarray*}
Each term in parenthesis is a sum over colourings of the blocks of $p$ or $q$, hence splits as a product over the blocks:
\begin{eqnarray*}
\sum_{\bo{x}\leqslant p}\bo{\chi}(\bo{x}^{d})\overline{\bo{\rho}}(\bo{x}^{u}) & = & \prod_{b\subset p}\left(\sum_{x\in G}\bo{\chi}_{b}(x)\overline{\bo{\rho}}^{b}(x)\right) \\
& = & \vert G\vert^{b(p)}\prod_{b\subset p}\delta(\bo{\chi}_{b}, \bo{\rho}^{b})
\end{eqnarray*}
The same formula holds for the second term in parenthesis (with $b(q)$ instead of $b(p)$) so that performing the multiplication  and using the fact that the tuples are admissible, we get
\begin{equation*}
X(\bo{g}, \bo{t}) = \vert G\vert^{b(p)+b(q)}\bo{\chi}(\bo{g})\overline{\bo{\rho}}(\bo{t}).
\end{equation*}
Gathering everything then yields the result.
\end{proof}

For the other category operations, the interplay with the maps $F_{p}$ is the same as for the maps $T_{p}$. More precisely, we have the following facts:

\begin{lem}\label{lem:propertiesfp}
Let us denote by $\circ$ the concatenation of tuples. Then,
\begin{itemize}
\item $F_{p}(\bo{\chi}, \bo{\rho})\otimes F_{q}(\bo{\chi}', \bo{\rho}') = F_{p\otimes q}(\bo{\chi}\circ\bo{\chi}', \bo{\rho}\circ\bo{\rho}')$.
\item $F_{p}(\bo{\chi}, \bo{\rho})^{*} = F_{p^{*}}(\bo{\rho}, \bo{\chi})$.
\end{itemize}
\end{lem}

Let us end this subsection with comments on the non-abelian case. The part concerning the space $\Mor(\V, \V)$ is easily generalised to this setting. If $\pi$ is an irreducible representation of $G$ and if $1\leqslant i\leqslant \dim(\pi)$, we can set
\begin{equation*}
P_{\pi, i} = \frac{1}{\vert G\vert}\sum_{g\in G}\pi(g)_{ii}M_{\vert^{g}_{e}}.
\end{equation*}
Denoting by $\Ir(G)$ a set of representatives of equivalence classes of irreducible representations of $G$, and by $\varepsilon$ the trivial representation of $G$, we obtain a generalisation of Lemma \ref{lem:basicprojections}.

\begin{lem}
The operators $(P_{\pi, i})_{\pi\in \Ir(G), 1\leqslant i\leqslant \dim(\pi)}$ are pairwise orthogonal self-adjoint projections summing to $1$. Moreover, for every non-trivial $\pi \in \Ir{G}$ the projection $P_{\pi}$ is minimal in $\Mor(\V, \V)$ and $P_{\varepsilon}$ is either minimal in $\Mor(\V, \V)$ (when $\CC(1,1)$ does not contain the two-block partition) or decomposes into two minimal projections (if $\CC(1,1)$  contains the two-block partition).
\end{lem}

\begin{proof}
The proof is the same as for Lemma \ref{lem:basicprojections} except that one uses the general Schur orthogonality relations instead of the orthogonality of characters.
\end{proof}

Note that $P_{\pi, i}$ and $P_{\pi, j}$ are equivalent for any $1\leqslant i, j\leqslant \dim(\pi)$, so that for each $\pi\in \Ir(G)$, we have a basic representation $u^{\pi}$ appearing with multiplicity $\dim(\pi)$ in $\V$. This result suggests to consider averaged morphisms where the characters are replaced by (diagonal) coefficients of irreducible representations of $G$. However, we were not able to prove a result similar to Proposition \ref{prop:compositionfp} for these more general objects. The reason for that is that a product of characters is again a character, so that in the proof of Proposition \ref{prop:compositionfp} we can always use the orthogonality relations. In the non-abelian case, we have to compute scalar products between products of coefficients, i.e. coefficients of tensor products of irreducible representations. Since these are not irreducible, the strategy breaks down at this point.

\subsection{The partition quantum group picture}

Thanks to the operators $F_{p}$, we can now give a description of the morphism spaces between tensor products of basic representations.

\begin{prop}\label{prop:samespan}
Let $\bo{\chi} = (\chi_{1}, \dots, \chi_{k})$ and $\bo{\rho} = (\rho_{1}, \dots, \rho_{l})$ be tuples of characters of $G$. Then,
\begin{equation*}
\Mor(u^{\chi_{1}}\otimes \dots\otimes u^{\chi_{k}}, u^{\rho_{1}}\otimes \dots\otimes u^{\rho_{l}}) = \Span\{F_{p}(\bo{\chi}, \bo{\rho}), p\in \CC\}.
\end{equation*}
\end{prop}

\begin{proof}
Recall that
\begin{equation*}
\Mor(u^{\chi_{1}}\otimes \dots\otimes u^{\chi_{k}}, u^{\rho_{1}}\otimes \dots\otimes u^{\rho_{l}}) = (P_{\chi_{1}}\otimes \dots\otimes P_{\chi_{k}})\Mor(\V^{\otimes k}, \V^{\otimes l})(P_{\rho_{1}}\otimes \dots\otimes P_{\rho_{k}}).
\end{equation*}
The right-hand side is generated by all the elements of the form
\begin{equation*}
R_{p}(\bo{\chi}, \bo{\rho}, \bo{s}, \bo{t}) = \sum_{\bo{g}, \bo{h}}\bo{\chi}(\bo{g})\bo{\rho}(\bo{h})M_{I(e, \bo{g})}M_{p(\bo{s}, \bo{t})}M_{I(e, \bo{h})},
\end{equation*}
where $I(e, \bo{g})$ is the partition $\vert^{\otimes k}$ coloured with $e$ on all lower points and $\bo{g}$ on the upper row. Computing the product of the averaged operators yields
\begin{equation*}
R_{p}(\bo{\chi}, \bo{\rho}, \bo{s}, \bo{t}) = \sum_{\bo{g}, \bo{h}}\bo{\chi}(\bo{g})\bo{\rho}(\bo{h})M_{p(\bo{s}\bo{g}, \bo{t}\bo{h}^{-1})} = \bo{\chi}(\bo{s}^{-1})\bo{\rho}(\bo{t})\sum_{\bo{g}', \bo{h}'}\bo{\chi}(\bo{g}')\bo{\rho}((\bo{h}')^{-1})M_{p(\bo{g}', \bo{h}'}) \in \C.F_{p}(\bo{\chi}, \bo{\rho}),
\end{equation*}
hence the result.
\end{proof}

This description is useful because it enables us to relate our construction with partition quantum groups as defined in \cite{freslon2014partition}. Let us briefly recall what we mean by this. A \emph{coloured partition} is a partition with the additional data of a colour (i.e. an element of a fixed set $\A$ called the \emph{colour set}) attached to each point. The category operations extend naturally to that setting with one extra rule: when a point is rotated from one row to the other, its colour is changed into the conjugate colour, where colour conjugation is a fixed involution on $\A$. A family of coloured partitions $\CC$ which is stable under the category operations is called a \emph{category of coloured partitions} and the set of partitions in $\CC$ with upper colouring $x_{1}\dots x_{k}$ and lower colouring $y_{1}\dots y_{l}$ is denoted by $\CC(x_{1}\dots x_{k}, y_{1}\dots y_{l})$. To any coloured partition $p$ we can associate the linear map $T_{p}$ as we did for partitions in $\CC^{G}$. Let $N\geqslant 1$ be an integer, let $\A$ be a colour set and let $\CC$ be a category of partitions coloured by $\A$. By \cite[Thm 3.2.7]{freslon2014partition}, there exists a unique compact quantum group $\G_{N}(\CC)$ together with representations $u^{x}$ for all $x\in\A$ such that for all $x_{1},\dots, x_{k}, y_{1},\dots, y_{l}\in \A$,
\begin{equation*}
\Mor(u^{x_{1}}\otimes \dots\otimes u^{x_{k}}, u^{y_{1}}\otimes \dots\otimes u^{y_{l}}) = \Span\{T_{p}, p\in\CC(x_{1}\dots x_{k}, y_{1}\dots y_{l})\}.
\end{equation*}

The link between this construction and partition wreath products will be given by the following examples of categories of partitions.

\begin{de}
Let $\Gamma$ be a discrete group and let $\CC$ be a category of (uncoloured) partitions. We define $\CC[\Gamma]$ to be the category of partitions coloured by $\Gamma$ (with colour conjugation given by the inverse) such that the underlying uncoloured partition is in $\CC$ and in each block the product of the colours of the upper row is equal to the product of the colours of the lower row as elements of $\Gamma$.
\end{de}

We are now ready to give a partition quantum group picture for partition wreath products.

\begin{thm}\label{thm:partitionqgroup}
Let $G$ be a finite abelian group and $\CC$ a category of partitions. There is an isomorphism of compact quantum groups $G\wr\G_{N}(\CC)\simeq \G_{N}(\CC[\h{G}])$.
\end{thm}

\begin{proof}
First note that setting, for $p\in P(k, l)$, $\widetilde{F}_{p}(\bo{\chi}, \bo{\rho}) = \vert G\vert^{(k+l)/2-b(p)}F_{p}(\bo{\chi}, \bo{\rho})$, we obtain the nicer formula
\begin{equation*}
\widetilde{F}_{p}(\bo{\chi}, \bo{\rho})\widetilde{F}_{q}(\bo{\chi}', \bo{\rho'}) = N^{\rl(p, q)}\delta(\bo{\rho}, \bo{\chi}')\widetilde{F}_{pq}(\bo{\chi}, \bo{\rho}').
\end{equation*}
Consider now the functor from the category of representations of $\G_{N}(\CC[\h{G}])$ to the category of representations of $G\wr \G_{N}(\CC)$ sending the object $u^{\chi(1)}\otimes \dots u^{\chi(k)}$ to $(\ii\otimes P_{\bo{\chi}})\V^{\otimes k}(\ii\otimes P_{\bo{\chi}})$ and the morphism $T_{p(\bo{\chi}, \bo{\rho})}$ to $\widetilde{F}_{p}(\bo{\chi}, \bo{\rho})$. By Proposition \ref{prop:compositionfp} and Lemma \ref{lem:propertiesfp}, this is a tensor functor. Because the conditions defining $\CC[\h{G}]$ precisely mean that we only consider partitions coloured with admissible pairs of tuples of characters, this functor is faithful and by Proposition \ref{prop:samespan} it is surjective, hence the result.
\end{proof}

The interest of Theorem \ref{thm:partitionqgroup} is that the representation theory of partition quantum groups can be studied using the general techniques of \cite{freslon2013representation}. These techniques adapt straightforwardly to the general setting of partition quantum groups as explained in \cite{freslon2014partition}.

\subsection{Examples}\label{subsec:examples}

We will now describe explicitly some examples. A consequence of Theorem \ref{thm:partitionqgroup} is that when $\CC$ is a category of non-crossing partitions which is \emph{block-stable} in the sense of \cite[Def 4.7]{freslon2013fusion} (see also \cite[Def 3.4]{banica2012noncommutative}, where the name "multiplicative" was used for the same concept), then $G\wr \G_{N}(\CC)$ is a \emph{free quantum group} by \cite[Thm 4.18]{freslon2013fusion}. We will not enter the details of the general theory here but simply explain how one can compute the fusion rules of such a quantum group. Given a compact quantum group $\G$, we will be interested in the \emph{fusion semiring} $R^{+}(\G)$ of $\G$, which is the semiring generated by the equivalence classes of irreducible representations, with addition and multiplication given by the direct sum and the tensor product.

\subsubsection{Free partition quantum groups}

The basic tool to describe the representation theory of free partition quantum groups is the notion of projective partition.

\begin{de}
A coloured partition $p$ is said to be \emph{projective} if $p^{2} = p = p^{*}$. Two projective partitions $p, q\in \CC$ are said to be \emph{equivalent in $\CC$} if there exists another partition $r\in \CC$ such that $r^{*}r = p$ and $rr^{*} = q$.
\end{de}

Let us denote by $W(\CC)$ the set of all projective partitions in $\CC$ which have only one block. This can be seen as a set of words $w$ on the colour set $\A$, precisely those for which $\CC(w, w)$ (i.e. those coloured partitions for which both the top and bottom row are coloured by the word $w$) contains a one-block partition. The equivalence relation for projective partitions induces an equivalence relation on $W(\CC)$: two words $w$ and $w'$ are equivalent if there is a one-block partition in $\CC(w, w')$. We will denote by $S(\CC)$ the quotient of $W(\CC)$ by this equivalence relation. The set $S(\CC)$ is called the \emph{fusion set} of $\CC$ and completely determines the representation theory of the associated quantum group. To explain how, let us first introduce two operations on $S(\CC)$ (given two words $w,v$ on the set $\A$, $w.v$ denotes their concatenation).

\begin{de}
Let $p\in \CC(w, w')$ and $q\in \CC(v, v')$ be one-block projective partitions and let $[p]$ and $[q]$ be their images in $S(\CC)$.
\begin{itemize}
\item If there exists a one-block partition $s\in \CC(w.v, w'.v')$, then we set $[p]\ast [q] = [s]$. Otherwise, we set $[p]\ast [q] = \emptyset$.
\item Let $\overline{p}$ be the partition obtained by rotating $p$ upside down (note that all the colours are conjugated in this process). We set $\overline{[p]} = [\overline{p}]$.
\end{itemize}
These operations are well-defined (i.e. do not depend on the choice of representatives) by \cite[Lem 4.13 and Lem 4.14]{freslon2013fusion}.
\end{de}

Let $F(\CC)$ be the free monoid on $S(\CC)$, i.e. the set of all words on $S(\CC)$. The previous operations extend to the abelian semigroup $\N[F(\CC)]$ in the following way: if $w_{1}\dots w_{n}, w'_{1}\dots w'_{k} \in F(\CC)$, then
\begin{itemize}
\item $\overline{w_{1}\dots w_{n}} = \overline{w_{n}}\dots \overline{w_{1}}$
\item $(w_{1}\dots w_{n})\ast (w'_{1}\dots w'_{k}) = w_{1}\dots (w_{n}\ast w'_{1})\dots w'_{k}$
\end{itemize}
the latter being set equal to $0$ whenever one of the two words is empty, or if $w_{n}\ast w_{1}' = \emptyset$. We can now turn $\N[F(\CC)]$ into a \emph{semiring} $(R^{+}(\CC), +, \otimes)$ by setting for any $w, w' \in F(\CC)$
\begin{equation*}
w\otimes w' = \sum_{w = az, w' = \overline{z}b} ab + a\ast b.
\end{equation*}

The link between $R^{+}(\CC)$ and $\G_{N}(\CC)$ is given by \cite[Thm 4.18]{freslon2013fusion}. This requires block-stability, an extra assumption on $\CC$ mentioned above, meaning that if $p\in \CC$ and if $b$ is a block of $p$, then $b\in \CC$.

\begin{thm}
Assume that $\CC$ is a block-stable category of non-crossing coloured partitions. Then, there is an isomorphism of fusion semirings $R^{+}(\G_{N}(\CC)) \simeq R^{+}(\CC)$ sending the representation $u^{x}$ to the class of the partition $\vert$ coloured with $x$ on both ends.
\end{thm}

\subsubsection{The permutation case}\label{subsec:permutationcase}

The case of partition wreath products by the quantum permutation group was already treated in Theorem \ref{thm:freewreath}: $G\wr S_{N}^{+}$ is isomorphic to $G\wr_{\ast}S_{N}^{+}$. This is true for any finite group $G$. If $G$ is abelian, then we can write it as $\h{\Gamma}$, where $\Gamma = \h{G}$ is the dual group. Then, Theorem \ref{thm:partitionqgroup} tells us that $\h{\Gamma}\wr_{\ast}S_{N}^{+}$ is the partition quantum group associated to the category of partitions $\CC[\Gamma]$. We therefore recover a particular case of the results of F. Lemeux in \cite{lemeux2013fusion}. Note that in that case, $S(\CC[\Gamma])$ is naturally isomorphic to $\Gamma$, with the conjugation given by the group inverse and the operation $\ast$ being the group operation.

\subsubsection{The orthogonal case}

Using the partition quantum group picture, it is easy to give an explicit description of partition wreath products by the quantum orthogonal group. In what follows below, $\hat{\ast}$ denotes the \emph{dual free product} of compact quantum groups introduced by Wang.

\begin{prop}\label{prop:abelianorthogonal}
Let $G$ be a finite abelian group and let $N\in \N$. Let $k = \vert\{\chi\in \h{G} : \chi= \overline{\chi}\}\vert$ and let $l = (\vert G\vert - k)/2$. Then $G \wr O_{N}^{+} \simeq (O_{N}^{+})^{\hat{\ast} k}\,\, \hat{\ast}\,\, (U_{N}^{+})^{\hat{\ast} l}$.
\end{prop}

\begin{proof}
We are considering the category of partitions $\CC[\h{G}] = NC_{2}[\h{G}]$ consisting of coloured non-crossing pairs. Consequently, the only one-block projective partitions are of the form $\vert$ coloured with a character $\chi$ on both ends. Let us denote by $p(\chi)$ this partition. Then, $\overline{[p(\chi)]} = [p(\overline{\chi})]$ and $[p(\chi)]\ast [p(\chi')] = \emptyset$. Let us order the self-conjugate characters of $G$ as $\chi_{1}, \dots, \chi_{k}$ and the non-self-adjoint ones as $\rho_{1}, \overline{\rho}_{1}\dots, \rho_{l}, \overline{\rho}_{l}$. There is a surjective $*$-homomorphism
\begin{equation*}
\Phi : C((O_{N}^{+})^{\hat{\ast} k} \hat{\ast} (U_{N}^{+})^{\hat{\ast} l}) \longrightarrow C(G \wr O_{N}^{+})
\end{equation*}
sending the fundamental representation of the $i$-th copy of $O_{N}^{+}$ to $u^{\chi_{i}}$ and the fundamental representation of the $j$-th copy of $U_{N}^{+}$ to $u^{\rho_{j}}$. Moreover, $\Phi$ induces an isomorphism between the fusion semirings. Since this isomorphism preserves the dimensions, $\Phi$ must be an isomorphism of compact quantum groups by \cite[Lem 5.3]{banica1999representations}, hence the result.
\end{proof}

\begin{rem}
One can also give a direct proof of this result using only the structure of the $(G, N)$-sudoku matrix and elementary manipulations using the orthogonal relations.
\end{rem}

Note that the partition wreath product $G\wr O_{N}^{+}$ only depends on the number of self-adjoint and non-self-adjoint characters of $G$. In particular, there is no canonical way to recover $C(G)$ as a subalgebra of $C(G\wr O_{N}^{+})$, in sharp contrast with the case of partition wreath products with $S_{N}^{+}$. The case of the classical orthogonal group is straightforward by abelianization.

\begin{cor}
Let $G$ be a finite abelian group and let $N\in \N$. Let $k = \vert\{\chi \in \h{G} : \chi= \overline{\chi}\}\vert$ and let $l = (\vert G\vert -k)/2$. Then, $G \wr O_{N} \simeq (O_N)^{\times k} \times (U_N)^{\times l}$.
\end{cor}

\begin{proof}
This follows from Proposition \ref{prop:abelianorthogonal} and Proposition \ref{prop:abelianization}.
\end{proof}

\subsubsection{The hyperoctahedral case}

We now turn to the hyperoctahedral case, which means that we will consider the category of partitions $\CC_{ev}$ consisting of all partitions with all blocks of even size. This was already treated in full generality in Theorem \ref{thm:hyperoctahedralgeneral} but we will recover the result in the abelian case in a completely different way. To describe the semiring $R^{+}(\CC[\h{G}])$ we use the following notation: if $w = w_{1}\dots w_{n}$ is a word on $\h{G}$, then $f(w)\in \h{G}$ denotes the product of the letters in $\h{G}$. With this, we can classify one-block projective partitions up to equivalence. For two words $w, w'$ on a colour set $\A$, let us denote by $\pi(w, w')$ the unique one-block partition in $P(w, w')$. Let us also denote by $\varepsilon\in \h{G}$ the trivial character of $G$.

\begin{lem}\label{lem:equivalencehhyperoctahedral}
Let $p\in NC_{ev}[\h{G}](w, w)$ be a one-block projective partition. If $\vert w\vert$ is even, then $p$ is equivalent to $\pi(f(w)\varepsilon, f(w)\varepsilon)$. Otherwise, it is equivalent to $\pi(f(w), f(w))$. Moreover, $\pi(f(w)\varepsilon, f(w)\varepsilon)$ and $\pi(f(w'), f(w'))$ are never equivalent for any words $w$ and $w'$.
\end{lem}

\begin{proof}
Assume first that $|w|$ is even. Then, $\pi(w, f(w)\varepsilon)$ is even and the product of the upper colours is equal to the product of the lower colours, hence it belongs to $\CC[\h{G}]$, giving the equivalence. Assume now that $|w|$ is odd. Then, $\pi(w, f(w))$ is in $\CC[\h{G}]$, hence the equivalence. Eventually, $\pi(f(w)\varepsilon, f(w)\varepsilon)$ and $\pi(f(w'), f(w'))$ are not equivalent because $\pi(f(w)\varepsilon, f(w))$, being odd, is never in $\CC[\h{G}]$.
\end{proof}

Let us set, for $\chi\in \h{G}$, $\chi^{-} = [\pi(\chi, \chi)]$ and $\chi^{+} = [\pi(\chi \varepsilon, \chi \varepsilon)]$. The previous lemma implies that $S(\CC[\h{G}]) = \{\chi^{-}: \chi\in \h{G}\}\coprod\{\chi^{+},: \chi\in G\}$.

\begin{lem}\label{lem:directproductz2}
There is a group isomorphism $\Psi : \h{G}\times \Z_{2} \rightarrow S(\CC[\h{G}])$ sending $(\chi, -1)$ to $\chi^{-}$ and $(\chi, 1)$ to $\chi^{+}$.
\end{lem}

\begin{proof}
The map $\Psi$ is clearly a bijection, so that we only need to check its compatibility with the group structure. By definition, $\overline{\chi^{-}} = \overline{\chi}^{-}$ and $\overline{\chi^{+}} = \overline{\chi}^{+}$ so that inverses are preserved. Moreover, we have the following rules coming from Lemma \ref{lem:equivalencehhyperoctahedral}:
\begin{equation*}
\chi^{-}\ast\rho^{-} = (\chi\rho)^{+}, \chi^{-}\ast\rho^{+} = (\chi\rho)^{-}, \chi^{+}\ast\rho^{-} = (\chi\rho)^{-}, \chi^{+}\ast\rho^{+} = (\chi\rho)^{+}.
\end{equation*}
These are the same as the group law on $\h{G}\times \Z_{2}$, hence the result.
\end{proof}

We can now complete an alternative proof of Theorem \ref{thm:hyperoctahedralgeneral} (for abelian $G$).

\begin{prop}\label{prop:abelianhyperoctahedral}
Let $G$ be a finite abelian group and let $N\in \N$. Then, $G\wr H_{N}^{+}$ is isomorphic to $(G\times \Z_{2})\wr_{\ast} S_{N}^{+}$.
\end{prop}

\begin{proof}
Recall from Subsection \ref{subsec:permutationcase} that $(\h{G}\times \Z_{2})\wr_{\ast} S_{N}^{+}$ is the free partition quantum group associated to $NC[\h{G}\times \Z_{2}]$ and that the associated fusion set $S$ is $\h{G}\times \Z_{2}$. Let us consider the $*$-homomorphism
\begin{equation*}
\Phi : C((\h{G}\times \Z_{2})\wr_{\ast} S_{N}^{+}) \longrightarrow C(G\wr H_{N}^{+})
\end{equation*}
sending $u^{(\chi, -1)}$ to $u^{\chi}$. At the level of the associated fusion sets, $\Phi$ induces the map $\Psi$ from Lemma \ref{lem:directproductz2}. Thus, the induced map $\widetilde{\Psi} : R^{+}((\h{G}\times \Z_{2})\wr_{\ast} S_{N}^{+})\rightarrow R^{+}(G\wr H_{N}^{+})$ is an isomorphism. Since it is dimension preserving, we can once again conclude by \cite[Lem 5.3]{banica1999representations} and by the observation that for finite abelian groups we have $G\approx \h{G}$.
\end{proof}

Again of course one can use abelianization to recover a particular case of Corollary \ref{cor:abelianizedhyperoctahedral}.

\section{Generalisations}\label{sec:generalisation}

In this final section, we will generalise the partition wreath product construction in two steps. The first step is quite natural and the results of Section \ref{sec:construction} can be adapted to some extent. The second step will only be sketched in the simplest case, which already gives interesting new examples connected to quantum isometry groups.

\subsection{Partition wreath product relative to an action}

The idea of this first step is to use an arbitrary action of the finite group $G$ on a finite space $X$ in the definition of the averaged operators, instead of the action of $G$ on itself. So let $\alpha : G \curvearrowright X$ be an action and let $\CC$ be a category of uncoloured partitions. Colouring the points of the partition of $\CC$ by elements of $X$ and applying the same averaging procedure as in Definition \ref{de:averagedoperators} using $\alpha$ instead of the group multiplication yields operators $M_{p}^{\alpha}$. We then set
\begin{equation*}
\Mor_{G\wr_{\alpha}}(k, l) = \Span\{M_{p}^{\alpha}, p\in \CC^{\A}(k, l)\}.
\end{equation*}

\begin{prop}
Let $N\geqslant 1$ be an integer. Then, there is a concrete C*-tensor category with duals $\mathfrak{C}(\CC, \alpha, N)$ whose objects are the positive integers and whose morphism sets are precisely $\Mor_{G\wr_{\alpha}}(k, l)$. The associated compact quantum group is called the \emph{partition wreath product of $G$ by $\G_{N}(\CC)$ relative to $\alpha$} and is denoted by $G\wr_{\alpha}\G_{N}(\CC)$.
\end{prop}

\begin{proof}
One first has to check that the sets of morphisms are stable under composition. This is done exactly as in Lemma \ref{lem:compositionaveraged} except that the scalar factor appearing in the product may not be the same. The stability under tensor products and rotations is clear.
\end{proof}

\begin{rem}\label{rem:faithfulaction}
Note that we can always assume the action to be faithful. Indeed, let $H$ be its kernel. Then, we have an action $\widetilde{\alpha}$ of $G/H$ on $X$ and
\begin{equation*}
G\wr_{\alpha}\G_{N}(\CC) = (G/H)\wr_{\widetilde{\alpha}}\G_{N}(\CC).
\end{equation*}
Let also remark that for any subgroup $H$ of $G$, there is a natural inclusion
\begin{equation*}
H\wr_{\alpha_{\vert H}}\G_{N}(\CC)\subset G\wr_{\alpha}\G_{N}(\CC),
\end{equation*}
which comes from the reverse inclusion of the associated C*-tensor categories.
\end{rem}

Some results of Section \ref{sec:construction} carry on to this setting with suitable modifications. In particular, the abelianization procedure works exactly as in Proposition \ref{prop:abelianization}. There is also a "fundamental matrix" description of $G\wr_{\alpha}\G_{N}(\CC^{\A})$ but it is more complicated; in particular we obtain a block-type splitting with respect to orbits of the action. This is the content of the next proposition.

\begin{prop}\label{prop:actionsudoku}
Let $B_{N}(\CC, \alpha)$ be the universal C*-algebra generated by a $N\vert X\vert \times N\vert X\vert$ matrix $\V$ such that
\begin{itemize}
\item $\V_{i, x}^{j, y} = 0$ for all $i, j$ if $x$ and $y$ are not in the same orbit.
\item The matrix $(\V_{i, g_{1}.x}^{j, g_{2}.x})_{1\leqslant i, j\leqslant N, g_{1}, g_{2}\in G}$ is unitary and satisfies the $\CC$-relations for each $x\in X$.
\item $\displaystyle\sum_{h\in \Stab(y)}\V_{i, th.x}^{j, s.x} = \sum_{g\in \Stab(x)}\V_{i, t.y}^{j, sg.y}$ for any $s, t\in G$ and $x, y\in X$.
\end{itemize}
Then, $B_{N}(\CC, \alpha)$ is isomorphic to $C(G\wr_{\alpha}\G_{N}(\CC))$ and the pullback of the coproduct is given by
\begin{equation*}
\D(\V_{i, x}^{j, y}) = \sum_{k=1}^{N}\sum_{z\in X} \V_{i, x}^{k, z}\otimes \V_{k, z}^{j, y}.
\end{equation*}
\end{prop}

\begin{proof}
The proof follows the same strategy as in Theorem \ref{thm:sudoku}. The relations satisfied by $\V$ can all be obtained by using partitions $p\in \CC$ coloured with one element $x\in X$ and partitions of the form $\vert^{y}_{x}$ for $x, y\in X$.

Let us first consider the operator $\widetilde{M}_{p, x}$ obtained by colouring all the points of $p$ with $x$ and then averaging. It acts on $\Span\{e_{i}^{g.x}, g\in G\}$, hence the fact that it is an intertwiner is equivalent to the $\CC$-relations for the entries of the matrix $(\V_{i, g_{1}.x}^{j, g_{2}.x})_{g_{1}, g_{2}\in G}$ (once we see it is unitary, which we observe below).

We now turn to $M_{\vert^{y}_{x}}$ and compute
\begin{equation*}
\V\circ(\Id\otimes M_{\vert^{y}_{x}})(1\otimes e_{i}^{z}) = \V\left(\sum_{g.y = z}1\otimes e_{i}^{g.x}\right) = \sum_{i=1}^{N}\sum_{v\in X}\sum_{g.y=z}\V_{i, g.x}^{j, v}\otimes e_{j}^{v}
\end{equation*}
and
\begin{equation*}
(\Id\otimes M_{\vert^{y}_{x}})\circ\V\left(1\otimes e_{i}^{z}\right) = (\Id\otimes M_{\vert^{y}_{x}})\left(\sum_{l=1}^{N}\sum_{w\in X}\V_{i, z}^{l, w}\otimes e_{l}^{w}\right) = \sum_{l=1}^{N}\sum_{w\in X}\V_{i, z}^{l, w}\sum_{h.y=w}e_{l}^{h.x}.
\end{equation*}
Applying $(1\otimes e_{j_{0}}^{v_{0}})$ for some $j_{0}, v_{0}$ yields
\begin{equation}\label{eq:roughcondition}
\sum_{v\in X}\sum_{g.y=z}\V_{i, g.x}^{j_{0}, v_{0}}\delta_{v, v_{0}} = \sum_{w\in X}\sum_{h.y=w}\V_{i, z}^{j_{0}, w}\delta_{h.x, v_{0}}.
\end{equation}
If $v_{0}$ is not in the orbit of $x$, the right-hand side vanishes and we are left with
\begin{equation*}
\sum_{g.y=z}\V_{i, g.x}^{j_{0}, v_{0}} = 0.
\end{equation*}
Specializing to $y=x$ eventually gives
\begin{equation*}
\vert\{g\in G, g.x=z\}\vert\V_{i, z}^{j_{0}, v_{0}} = 0
\end{equation*}
and the first relation is proved (noticing that $z$ can be any element in the orbit of $x$), showing also that the matrices appearing in the second relation are unitary. Let us now assume that $v_{0}$ is in the orbit of $x$ and that $z$ is in the orbit of $y$ (otherwise both sides of \eqref{eq:roughcondition} vanish). There exist $t, s\in G$ such that $v_{0} = t.x$ and $z = t.y$ so that
\begin{align*}
\sum_{\underset{h.y=t.y}{g\in G}}\V_{i, h.x}^{j_{0}, s.y} & = \sum_{w\in X}\sum_{\underset{g.y=w, g.x=s.x}{g\in G}}\V_{i, t.y}^{j_{0}, w} \\
\sum_{\underset{h'.y=y}{h'\in G}}\V_{i, th'.x}^{j_{0}, s.x} & = \sum_{w\in X}\sum_{\underset{g'.x=x, sg'.y=w}{g'\in G}}\V_{i, t.y}^{j_{0}, w} \\
& = \frac{1}{\vert \Stab(y)\vert}\sum_{k\in G}\sum_{\underset{g'.x=x, sg'.y=k.y}{g'\in G}}\V_{i, t.y}^{j_{0}, k.y},
\end{align*}
using the changes of variables $h'=t^{-1}h$ and $g'=s^{-1}g$ together with the fact that $w$ must be in the orbit of $y$, hence has the form $k.y$ for some $k \in G$. Renaming $h', g', j_{0}$ as $h, g, j$ yields
\begin{equation*}
\sum_{h\in \Stab(y)}\V_{i, th.x}^{j, s.x} = \frac{1}{\vert \Stab(y)\vert }\sum_{g\in \Stab(x)}\sum_{\underset{k^{-1}sg.y=y}{k\in G}}\V_{i, t.y}^{j, k.y}.
\end{equation*}
The condition on $k$ is equivalent to $k\in sg.\Stab(y)$, i.e. $k.y=sg.y$ and the sum over $k$ simplifies with the factor $\vert \Stab(y)\vert^{-1}$ , giving the third relation.
\end{proof}

Let us now give a few examples of this construction:
\begin{itemize}
\item If $G$ acts on the set $X$ trivially, then $\V_{i, x}^{j, x} = \V_{i, y}^{j, y}$ for any $x, y\in X$ and all the other coefficients are zero. Thus, we simply have $G\wr_{\alpha}\G_{N}(\CC)\approx\G_{N}(\CC)$ with fundamental representation $\V = U^{\oplus \vert X\vert}$.
\item If the action $\alpha$ is free then all the stabilizers reduce to the neutral element, hence we get
\begin{equation*}
\V_{i, t.x}^{s.x} = \V_{i, t.y}^{j, s.y}.
\end{equation*}
In other words, $\V_{i, a}^{j, b}$ only depends on the unique element of $G$ sending $a$ to $b$. In particular, the matrices corresponding to different orbits are the same, thus the resulting quantum group is the same as the one corresponding to the restriction of the action to one orbit. This is equivalent to $G$ acting on itself, hence we get $G\wr\G_{N}(\CC)$.
\item If the action $\alpha$ is transitive, then we can assume that $X = G/H$ (with $\alpha$ being the translation action) for some subgroup $H$ of $G$. The kernel of this action is the normal core $Co(H)$ of $H$ in $G$ (i.e. the largest normal subgroup of $G$ which is contained in $H$) so that by Remark \ref{rem:faithfulaction}, the resulting quantum group is $(G/Co(H))\wr\G_{N}(\CC)$.
\end{itemize}

Decomposing a general action $\alpha : G \curvearrowright X$ as a disjoint union of transitive actions on the orbits and using the last point, it would be enough to understand the case of a faithful action. However, the defining equations for $\V$ do not simplify in that case as in the free case and we were not able to obtain a general structure theorem.

\subsection{Coloured partition wreath products}

The second way of generalising our earlier considerations is based on putting the whole construction in the framework of partition quantum groups by using coloured partitions from the very beginning. To do this, we will assume that the set on which $G$ acts is also connected to the category of partitions $\CC$, specifically being its colour set. So let $G$ be a finite group, let $\A$ be a colour set with a fixed involution, let $\CC^{\A}$ be a category of $\A$-coloured partitions and consider an action $\alpha : G \curvearrowright \A$. In order to build a C*-tensor category, we need to put a constraint on the action.

\begin{de}
An action $\alpha : G \curvearrowright \A$ is said to be a \emph{coloured action} if for any $g\in G$ and any $x\in \A$,
\begin{equation*}
\overline{g.x} = g.\overline{x}.
\end{equation*}
\end{de}

Given a coloured action, we can define operators $M_{p}^{\alpha}$ exactly as before and use them to construct a C*-tensor category $\mathfrak{C}(\CC, \alpha, N)$ (the coloured action condition is needed here for the stability under rotations) yielding the \emph{partition wreath product of $G$ by $\G_{N}(\CC^{\A})$ relative to $\alpha$}, denoted by $G\wr_{\alpha}\G_{N}(\CC^{\A})$. Contrary to the previous subsection, it is not straightforward to extend the fundamental matrix picture to this setting. This would require several matrices $\V^{x}$ indexed by elements of $\A$ in order to be able to make sense of the "$\CC^{\A}$-relations". Moreover, the $G$-invariance condition is quite unclear. We will not develop the full theory here but simply give an example illustrating what the general fundamental matrix picture may look like. This example will be the simplest possible, i.e. the colour set is $\A = \{\circ, \bullet\}$ with $\overline{\circ} = \bullet$ and $G = \Z_{2} = \{-1, 1\}$ acts by $-1.\circ = \bullet$, which is a coloured action. Then, the $\CC^{\circ, \bullet}$-relations are simply relations between $\V$ and its adjoint. Here is how the fundamental matrix should look like.

\begin{de}
Let $\CC^{\circ, \bullet}$ be a category of coloured partitions. We define $B_{N}(\CC^{\circ, \bullet}, \alpha)$ as the universal C*-algebra generated by a $2N \times 2N$ matrix $\V$ satisfying the $\CC^{\circ, \bullet}$-relations and such that for any $1\leqslant i, j\leqslant N$, any $x, y\in \{\circ, \bullet\}$,
\begin{equation*}
(\V_{i, x}^{j, y})^{*} = \V_{i, \overline{x}}^{j, \overline{y}}.
\end{equation*}
Endowed with the coproduct defined by
\begin{equation*}
\D(\V_{i, x}^{j, y}) = \sum_{k=1}^{N}\sum_{z\in X} \V_{i, x}^{k, z}\otimes \V_{k, z}^{j, y},
\end{equation*}
it is a compact quantum group.
\end{de}

We will prove in a particular case that $B_{N}(\CC^{\circ, \bullet}, \alpha)$ is indeed the partition wreath product relative to $\alpha$. Let $NC^{\circ, \bullet}_{s}$ be the category of non-crossing coloured partitions such that in each block the difference between the number of white and black points on each row is equal modulo $s$. We define similarly the category of partitions $NC^{\circ, \bullet}_{\infty}$ by the convention that equality module $\infty$ is just equality. The partition quantum group associated to $NC^{\circ, \bullet}_{s}$ is denoted by $H_{N}^{s+}$ and called a \emph{quantum reflection group}. These were introduced by T. Banica and R. Vergnioux in \cite{banica2009fusion}.

\begin{prop}\label{prop:abelianreflection}
Let $N$ and  $1\geqslant s\geqslant +\infty$ be integers. Then, $(B_{N}(\CC_{s}^{\circ, \bullet}, \alpha), \Delta)$ is isomorphic to $\Z_{2}\wr_{\alpha}H_{N}^{s+}$.
\end{prop}

\begin{proof}
In fact, we will prove that $(B_{N}(\CC^{\circ, \bullet}, \alpha), \Delta)$ is isomorphic to $(\widehat{\Z}_{s}\rtimes\Z_{2})\wr S_{N}^{+}$, where $\Z_{2}$ acts on $\widehat{\Z}_{s}$ by inversion. By \cite[Thm 3.5]{banica2011quantum} and \cite[Rem 4.1]{mandal2015quantum}, this implies the result. We identify the colour set $\A$ with $\Z_{2}$ via $\circ\mapsto 1$ and $\bullet\mapsto -1$. Let $\V$ be the $2N\times 2N$ unitary matrix generating $C(\Z_{2}\wr_{\alpha}H_{N}^{s+})$ and recall that the $NC^{\circ, \bullet}_{s}$-relations are those making $\V$ unitary and such that
\begin{equation*}
(\V_{i, x}^{j, y})^{s} = \V_{i, x}^{j, y}(\V_{i, x}^{j, y})^{*}
\end{equation*}
is a projection for each $i,j=1,\ldots, N$ and $x,y \in \Z_2$. For a character $\chi$ on $\Z_{s}$ we set
\begin{equation*}
Q_{i, x}^{j, y}(\chi) =  \sum_{k=0}^{s-1}\chi(k)(\V_{i, x}^{j, y})^{k},
\end{equation*}
where again  $i,j=1,\ldots, N$ and $x,y \in \Z_2$.  The properties of the coefficients of $\V$ (each of these satisfies the equalities $(a^k)^* = a^{s-k}$, $k=1, \ldots, s$) imply that each element
$Q_{i, x}^{j, y}(\chi)$ is a self-adjoint projection.
This defines $2N\times 2N$ matrices $Q(\chi)$ which we can now use as blocks to build a $2sN\times 2sN$ matrix $\W = (Q_{(x, \chi), (y, \rho)})_{x,y \in \Z_2, \chi, \rho \in \widehat{\Z}_{s}}$, where
\begin{equation*}
(Q_{(x, \chi), (y, \rho)})_{i, j} = Q_{i, x}^{j, y}(\overline{\chi}\rho).
\end{equation*}
The matrix $\W$ is by construction a magic unitary and its blocks are indexed by $\A\times \widehat{\Z}_{s}$. Let us denote by $\varphi : \Z_{2}\rightarrow \rm{Aut}(\widehat{\Z}_{s})$ the homomorphism such that $\varphi_{-1}(\chi) = \overline{\chi}$. In the semi-direct product $\widehat{\Z}_{s}\rtimes_{\varphi}\Z_{2}$, we have
\begin{equation*}
(x, \chi)^{-1}(y, \rho) = (x^{-1}, \varphi_{x^{-1}}(\overline{\chi}))(y, \rho) = (x^{-1}y, \varphi_{x^{-1}}(\overline{\chi}\rho).
\end{equation*}
If $x=1$, then
\begin{equation*}
(Q_{(x, \chi), (y, \rho)})_{i, j} = Q_{i, 1}^{j, y}(\overline{\chi}\rho) = Q_{i, 1}^{j, 1^{-1}y}(\varphi_{1^{-1}}(\overline{\chi}\rho)) = (Q_{(1, \varepsilon), (x, \chi)^{-1}(y, \rho)})_{i, j}.
\end{equation*}
If $x=-1$, then
\begin{eqnarray*}
Q_{i, x}^{j, y}(\overline{\chi}\rho) & = & \sum_{k=0}^{s-1}\overline{\chi}(k)\rho(k)(\V_{i, -1}^{j, y})^{k} \\
& = & \sum_{k=0}^{s-1}\overline{\chi}(k)\rho(k)(\V_{i, 1}^{j, -y})^{k*} \\
& = & \sum_{k=0}^{s-1}\overline{\chi}(k)\rho(k)(\V_{i, 1}^{j, -y})^{s-k} \\
& = & \sum_{k=0}^{s-1}\overline{\rho}(k)\chi(k)(\V_{i, 1}^{j, -y})^{k} \\
& = & Q_{i, 1}^{j, -y}(\overline{\rho}\chi) \\
& = & Q_{i, 1}^{j, x^{-1}y}(\varphi_{x^{-1}}(\overline{\chi}\rho)).
\end{eqnarray*}
In both cases, the $\widehat{\Z}_{s}\rtimes_{\varphi}\Z_{2}$-invariance condition is satisfied by $\W$ so that using Proposition \ref{thm:sudoku} we get the announced isomorphism.
\end{proof}

Note that the above quantum groups appeared as quantum isometry groups of the free products of finite cyclic groups in \cite{banica2012quantum} and \cite{goswami2014quantum}. There is also a classical version of this result.

\begin{cor}
Let $N, s\geqslant 1$ be integers. Then, $\Z_{2}\wr_{\alpha}H_{N}^{s}$ is isomorphic to $(\widehat{\Z}_{2}\rtimes \Z_{2})\wr S_{N}$, where $\Z_{2}$ acts on $\widehat{\Z}_{s}$ by inversion.
\end{cor}

\begin{proof}
This follows from Proposition \ref{prop:abelianreflection} by the abelianization procedure.
\end{proof}

These results recover the particular case $s=2$ treated in Theorem \ref{thm:hyperoctahedralgeneral} since the action by inversion then becomes trivial. In view of Proposition \ref{prop:abelianreflection}, we can ask a more general question. Let $\Gamma$ be a discrete group and consider a finite group $G$ acting on a symmetric generating set $\Lambda$ of $\Gamma$. If we consider $\Lambda$ as a colour set with involution given by the inverse, then the action is coloured if and only if $(\alpha_{g}(\gamma))^{-1} = \alpha_{g}(\gamma^{-1})$ for all $g\in G$ and $\gamma\in\Gamma$.

\begin{question}
Is there an isomorphism $G\wr_{\alpha}(\widehat{\Gamma}\wr_{\ast}S_{N}^{+}) \simeq (G\ltimes_{\alpha}\widehat{\Gamma})\wr_{\ast}S_{N}^{+}$ for a coloured action as described above?
\end{question}

Let us conclude by mentioning a broader problem. We have mainly studied partition wreath products for categories of partitions which are either noncrossing or contain the simple crossing $\crosspart$. We know from the classification of orthogonal easy quantum groups \cite{raum2013full} that there are many more examples. It would be interesting to see what happens when they are used as the right-hand side of a partition wreath product. In particular, the so-called \emph{half-liberated} quantum groups may produce interesting new phenomena.

\bibliographystyle{amsplain}
\bibliography{../../quantum}

\providecommand{\bysame}{\leavevmode\hbox to3em{\hrulefill}\thinspace}
\providecommand{\MR}{\relax\ifhmode\unskip\space\fi MR }
\providecommand{\MRhref}[2]{%
  \href{http://www.ams.org/mathscinet-getitem?mr=#1}{#2}
}
\providecommand{\href}[2]{#2}
\begin{thebibliography}{10}

\bibitem{banica1999representations}
T.~Banica, \emph{{Representations of compact quantum groups and subfactors}},
  J. Reine Angew. Math. \textbf{509} (1999), 167--198.

\bibitem{banica2007free}
T.~Banica and J.~Bichon, \emph{{Free product formul\ae{} for quantum
  permutation groups}}, J. Inst. Math. Jussieu \textbf{6} (2007), no.~03,
  381--414.

\bibitem{banica2011quantum}
T.~Banica, J.~Bichon, and S.~Curran, \emph{Quantum automorphisms of twisted
  group algebras and free hypergeometric laws}, Proc. Amer. Math. Soc
  \textbf{139} (2011), 3961--3971.

\bibitem{banica2012quantum}
T.~Banica and A.~Skalski, \emph{Quantum isometry groups of duals of free powers
  of cyclic groups}, Int. Math. Res. Not. \textbf{2012} (2012), no.~9,
  2094--2122.

\bibitem{banica2012noncommutative}
T.~Banica, A.~Skalski, and P.~So{\l}tan, \emph{{Noncommutative homogeneous
  spaces : the matrix case}}, J. Geom. Phys. \textbf{62} (2012), no.~6,
  1451--1466.

\bibitem{banica2009liberation}
T.~Banica and R.~Speicher, \emph{{Liberation of orthogonal Lie groups}}, Adv.
  Math. \textbf{222} (2009), no.~4, 1461--1501.

\bibitem{banica2009fusion}
T.~Banica and R.~Vergnioux, \emph{Fusion rules for quantum reflection groups},
  J. Noncommut. Geom. \textbf{3} (2009), no.~3, 327--359.

\bibitem{bichon2004free}
J.~Bichon, \emph{Free wreath product by the quantum permutation group}, Algebr.
  Represent. Theory \textbf{7} (2004), no.~4, 343--362.

\bibitem{bloss2003g}
M.~Bloss, \emph{{$G$-colored partition algebras as centralizer algebras of
  wreath products}}, J. Algebra \textbf{265} (2003), no.~2, 690--710.

\bibitem{daws2012closed}
M.~Daws, P.~Kasprzak, A.~Skalski, and P.M. So{\l}tan, \emph{Closed quantum
  subgroups of locally compact quantum groups}, Adv. Math. \textbf{231} (2012),
  3473--3501.

\bibitem{freslon2013fusion}
A.~Freslon, \emph{{Fusion (semi)rings arising from quantum groups}}, J. Algebra
  \textbf{417} (2014), 161--197.

\bibitem{freslon2014partition}
\bysame, \emph{{On the partition approach to Schur-Weyl duality and free
  quantum groups}}, Arxiv preprint arXiv:1409.1346 (2014).

\bibitem{freslon2013representation}
A.~Freslon and M.~Weber, \emph{{On the representation theory of easy quantum
  groups}}, J. Reine Angew. Math. (2015).

\bibitem{goswami2014quantum}
D.~Goswami and A.~Mandal, \emph{{Quantum isometry group of dual of finitely
  generated discrete groups and quantum groups}}, arXiv preprint
  arXiv:1408.5683 (2014).

\bibitem{parvathi2004g}
A.J. Kennedy and M.~Parvathi, \emph{{$G$-vertex colored partition algebras as
  centralizer algebras of direct products}}, Comm. Algebra \textbf{32} (2004),
  no.~11, 4337--4361.

\bibitem{lemeux2013fusion}
F.~Lemeux, \emph{{Fusion rules for some free wreath product quantum groups and
  applications}}, J. Funct. Anal. \textbf{267} (2014), no.~7, 2507--2550.

\bibitem{lemeux2014free}
F.~Lemeux and P.~Tarrago, \emph{{Free wreath product quantum groups : the
  monoidal category, approximation properties and free probability}}, Arxiv
  preprint arXiv:1411.4124.

\bibitem{mandal2015quantum}
A.~Mandal, \emph{{Quantum isometry groups of dual of finitely generated
  discrete groups II}}, Arxiv preprint arXiv:1504.02240 (2015).

\bibitem{neshveyev2014compact}
S.~Neshveyev and L.~Tuset, \emph{{Compact quantum groups and their
  representation categories}}, Specialized courses, SMF, 2013.

\bibitem{nica2006lectures}
A.~Nica and R.~Speicher, \emph{{Lectures on the combinatorics of free
  probability}}, Lecture note series, vol. 335, London Mathematical Society,
  2006.

\bibitem{raum2013full}
S.~Raum and M.~Weber, \emph{{The full classification of orthogonal easy quantum
  groups}}, Arxiv preprint arXiv:1312.3857 (2013).

\bibitem{wang1998quantum}
S.~Wang, \emph{Quantum symmetry groups of finite spaces}, Comm. Math. Phys.
  \textbf{195} (1998), no.~1, 195--211.

\bibitem{woronowicz1987compact}
S.L. Woronowicz, \emph{{Compact matrix pseudogroups}}, Comm. Math. Phys.
  \textbf{111} (1987), no.~4, 613--665.

\bibitem{woronowicz1988tannaka}
\bysame, \emph{{Tannaka-Krein duality for compact matrix pseudogroups. Twisted
  SU(N) groups}}, Invent. Math. \textbf{93} (1988), no.~1, 35--76.

\bibitem{woronowicz1995compact}
\bysame, \emph{{Compact quantum groups}}, Sym{\'e}tries quantiques (Les
  Houches, 1995) (1998), 845--884.

\end{thebibliography}

\end{document}